%% file: ex_article.tex
\documentclass[review,onefignum,onetabnum]{siamart190516}

\input{ex_shared}
\ifpdf
\hypersetup{
  pdftitle={Accelerating Primal-dual Methods for Regularized Markov Decision Processes},
  pdfauthor={H. Li, H. Yu, L. Ying, and I. Dhillon}
}
\fi

\begin{document}
\nolinenumbers
\maketitle

\begin{abstract}
  Entropy regularized Markov decision processes have been widely used in reinforcement learning.
  This paper is concerned with the primal-dual formulation of the entropy regularized problems.
  Standard first-order methods suffer from slow convergence due to the lack of strict convexity and
  concavity. To address this issue, we first introduce a new quadratically convexified primal-dual
  formulation. The natural gradient ascent descent of the new formulation enjoys a global convergence
  guarantee and exponential convergence rate. We also propose a new interpolating metric that
  further accelerates the convergence significantly. Numerical results are provided to demonstrate
  the performance of the proposed methods under multiple settings.
\end{abstract}



\section{Introduction}\label{sec:intro}

\subsection{Setup}

Consider an infinite-horizon Markov decision process (MDP)
\cite{bellman2015applied,sutton2018reinforcement,puterman2014markov}
$\mathcal{M}=(S,A,P,r,\gamma)$, where $S$ is a set of states of the Markov chain and $A$ is a set of
actions. $P$ is a transition probability tensor with $P_{ass'}$ being the probability of
transitioning from state $s$ to state $s'$ when taking action $a$, $r$ is a reward matrix with
$r_{sa}$ being the reward obtained when taking action $a$ at state $s$, and $\gamma\in(0,1)$ is the
discount factor. In this paper, we assume that the state space $S$ and the action space $A$ are
finite.

A policy $\pi$ is a randomized strategy over the actions at each state, i.e., for each state $s$,
$\pi_{sa}$ is the probability of choosing action $a$ at $s$. For a given policy, the value function
$v_\pi \in\R^{|S|}$ is a vector defined as
\begin{equation}
  (v_{\pi})_s := \mathbb{E} \sum_{k=0}^\infty \left(\gamma^k r_{s_k a_k}\mid s_0 = s\right),
\end{equation}
where the expectation is taken over all possible trajectories $\{(s_k,a_k)\}_{k\ge 0}$ starting from
$s_0=s$ following the policy $\pi$. The value function $v_\pi$ satisfies the well-known Bellman
equation \cite{bellman2015applied}
\begin{equation}\label{eq:vanillabellman}
  (I-\gamma P_\pi)v_\pi = r_\pi, 
\end{equation}
where $(P_\pi)_{ss'} := \sum_{a\in A}\pi_{sa}P_{ass'}$, $(r_\pi)_{s} :=\sum_{a\in A}\pi_{sa}r_{sa}$, and $I$ is
the identity operator. In a Markov decision problem, the goal is to find the optimal policy
$\pi^*$ such that
\[
v_{\pi^*}(s)\geq v_\pi(s), \quad \forall s\in S,
\]
for any other policy $\pi$. The corresponding optimal value function $v_{\pi^*}$ will also be
referred to as $v^*$ in this paper. The existence of $v^*$ and $\pi^*$ is guaranteed by the theory of
MDP \cite{puterman2014markov}.

In recent studies, entropy regularization has been widely used in MDP problems to encourage
exploration and enhance the robustness
\cite{neu2017unified,dai2018sbeed,geist2019theory,ahmed2019understanding,agarwal2020optimality,mei2020global,cen2020fast,zhan2021policy}.
With the entropy regularization, the value function is defined by
\begin{equation}\label{eq:regVfn}
  (v_{\pi})_s := \mathbb{E} \sum_{k=0}^\infty \left(\gamma^k (r_{s_k a_k}- \tau \log \pi_{s_k a_k})\mid s_0 = s\right),
\end{equation}
where $\tau > 0$ is the regularization coefficient.  $v_\pi$ satisfies the regularized Bellman
equation
\begin{equation}\label{eq:regbellman}
  (I-\gamma P_\pi)v_\pi = r_\pi - \tau h_\pi, 
\end{equation}
where $h_\pi$ is a vector in $\R^{|S|}$ with each entry $(h_\pi)_s$ given by the negative Shannon
entropy of $(\pi_{sa})_{a\in A}$
\[
(h_\pi)_s = \sum_{a\in A} \pi_{sa}\log\pi_{sa}.
\]
Here we overload the notation $v_\pi$ for the regularized value function and for the rest of the
paper $v_\pi$ shall always denote the regularized value function \eqref{eq:regVfn} unless otherwise
specified.  For the entropy regularized MDP (see \cite{geist2019theory}), there exists a unique
optimal policy $\pi^*$, such that \eq{\label{eq:optimality} v^*(s) := v_{\pi^*}(s)\geq v_\pi(s),
  \quad \forall s\in S,} for any other policy $\pi$.

Without loss of generality, {\bf the reward $r_{sa}$ is assumed to be nonnegative} throughout this
paper. This can be guaranteed by adding to the rewards a sufficiently large constant $C$. Note that
such a uniform shift keeps the optimal policy $\pi^*$ unchanged and shifts $v^*$ by a constant
$\frac{C}{1-\gamma}$.

\subsection{Primal-dual formulation}

Entropy regularized MDPs enjoy regularized linear programming formulations, in the primal, dual, and primal-dual
forms. In this paper, we are concerned with the primal-dual formulation (see, for example,
\cite{neu2017unified,ying2020note}):
\[
\min_{v\in\R^{|S|}} \max_{u\in \R^{|S|\times|A|}} \sum_{s\in S} e_s v_s + \sum_{s\in S,a\in A} u_{sa} (r_{sa}-((I-\gamma P_a)v)_s) - \tau \sum_{s\in S,a\in A} u_{sa} \log(u_{sa}/\bu_s),
\]
where $\bu_s:=\sum_{a\in A} u_{sa}$. The policy $\pi$ is related to $u$ via the relationship
\[
\pi_{sa} = u_{sa}/\bu_s.
\]
The main advantage of working with the primal-dual formulation is that the transition matrix $P_a$
appears linearly in the objective function of the primal-dual problem. This linearity brings an
important benefit when a stochastic gradient method is used to solve the primal-dual formulation: an
unbiased estimator of the transition matrix $P_a$ guarantees an unbiased estimator for the
gradient-based update rule. This avoids the famous double-sampling problem
\cite{sutton2018reinforcement} that affects any formulation that performs a nonlinear operation to
the transition matrix $P_a$. Examples of these affected formulations include the primal formulation,
where a nonlinear max or exponentiation operator is applied to $P_a$, and the dual formulation, where
the inverse of $I - \gamma P_\pi$ is needed. From this perspective, the primal-dual formulation
is convenient in the model-free setting, where the transition probability tensor can only be
estimated from samples and is thus inherently noisy.

In what follows, we shall simplify the notation by denoting $K_a=I-\gamma P_a$ and $K_\pi = I-\gamma
P_\pi$. With this simplification, the primal-dual problem can be rewritten more compactly as
\begin{equation} \label{eq:linearpd}
  \min_v \max_u \sum_s e_s v_s + \sum_{sa} u_{sa} (r_{sa}-(K_av)_s) - \tau
  \sum_{sa} u_{sa} \log(u_{sa}/\bu_s).
\end{equation}

Though theoretically appealing, the primal-dual formulation \eqref{eq:linearpd} often poses
computational challenges because it is a minimax optimization. Newton-type methods are often
impractical to apply because either $P_a$ is only accessible via samples or its size is too large
for practical inversion. A close look at the objective function of \eqref{eq:linearpd} suggests that
it is linear with respect to both the value function $v$ and the dual variable $u$ in the radial
direction. This lack of strict convexity/concavity makes it difficult for the first-order methods to
converge.

\subsection{Contributions}
To overcome this difficulty, this paper proposes a {\em quadratically convexified} reformulation of
\eqref{eq:linearpd} that shares the same solution with \eqref{eq:linearpd} and an {\em
  interpolating} natural gradient ascent descent method that significantly speeds up the
convergence. More specifically, the main contributions of this paper are listed as follows:
\begin{itemize}
\item We propose a new quadratically convexified primal-dual formulation in which the linear
  weighted sum $e^\T v$ of \eqref{eq:linearpd} is replaced with a quadratic term $\frac{\alpha}{2}
  \|v\|^2$. The surprising feature is that the solution $(v^*, u^*)$ is unchanged and is independent
  of the hyperparameter $\alpha>0$. We prove that the vanilla natural gradient ascent descent (NGAD)
  of this quadratically convexified problem enjoys a Lyapunov function \cite{lyapunov1992general}
  and converges linearly. To the best of our knowledge, this is the first quadratically convexified
  primal-dual formulation of Markov decision problems.
  
\item We propose an interpolating natural gradient ascent descent (INGAD) by introducing a new
  interpolating metric for the $u$ variable. The corresponding Lyapunov function is constructed and
  the convergence of the new dynamics is proved. The acceleration is verified by numerical tests
  under multiple settings.
\end{itemize}

\subsection{Related work}
Regarding the primal-dual formulation, the first primal-dual learning algorithm is given in
\cite{wang2016online}. A follow-up work \cite{wang2020randomized} leverages the binary-tree data
structure and adaptive importance sampling techniques to reduce the complexity. The convergence
result for these two papers is however only for the average of all the policies rather than the
policy obtained in the last iteration. In these papers, no regularization is used in the formulation
and no preconditioner is used in the iterative update scheme. As a comparison, the current paper
proves a last-iteration convergence result with the help of the Lyapunov method and entropy
regularization, and derives an interpolating metric that accelerates the algorithm.  Various studies
have been carried out following the primal-dual formulation in \cite{wang2016online}. For example, a
modified form with the $Q$-function is proposed in \cite{lee2018stochastic}, and the corresponding
primal-dual type algorithm is derived. An extension to the infinite-horizon average-reward setting
is provided in \cite{wang2017primal}, but only the average-case convergence result is given. A later
work \cite{cho2017deep} further extended this method to the function approximation setting. A
comprehensive review of the primal-dual methods in the average reward setting is given in a recent
thesis \cite{gong2021primal}, and a generalization to the general utility maximization formulation is
provided. The primal-dual method has also been used to find risk-sensitive policies, for example, in
\cite{zhang2021cautious}, where a risk function is integrated into the primal-dual objective through
the dual variable. In the optimization literature, the primal-dual formulation is often called the
saddle point problem: for example, \cite{serrano2020faster} considers a linear relaxation version of
the saddle-point problem in \cite{wang2017primal} to address large-scale problems. However, it is
worth noting that no (entropy) regularization is used in the papers mentioned above, which is able
to make the landscape of the optimization problem smoother and is thus a crucial element of recent
linear convergence results \cite{cen2020fast,li2021quasinewton, lan2023policy}. Linear convergence results can be developed with the presence of precondition. For example, in \cite{khodadadian2021linear}, the authors show that the natural policy gradient method with an exact evaluation of the gradient has a linear convergence rate after sufficiently many gradient steps, where the convergence rate relies on an advantage function gap. Without regularization or preconditioners,
gradient-type methods can take exponential time to converge \cite{li2021softmax}.

Besides the primal-dual formulations, the discussion below briefly touches on the primal and the
dual formulations. For the entropy regularized Markov decision process, the primal formulation
\cite{ying2020note} takes the form
\begin{equation}
  \label{eq:primal}
  v_s = \tau \log\left(\sum_{a} \exp\left( \frac{r_{sa}+\g\sum_{s'} P_{ass'}v_{s'}}{\tau},\right)\right),
\end{equation}
which leads to a value iteration algorithm. Let
$\varphi(v):\mathbb{R}^{|S|}\rightarrow\mathbb{R}^{|S|}$ be the fixed-point map such that
$\varphi(v)_s = \tau \log\left(\sum_{a\in A} \exp\left( \tau^{-1}(r_{sa}+\g\sum_{s'\in S}
P_{ass'}v_{s'}) \right)\right)$. By calculating the derivative matrix, we have
\[
\|D\varphi(v)\|_{\infty} = \underset{s}{\max} \sum_{s'}|(D\varphi(v))_{ss'}| 
= \underset{s}{\max} \sum_{s'}\frac{\gamma\sum_{a}  P_{ass'}\exp\left( \tau^{-1}(r_{sa}+\g\sum_{s''} P_{ass''}v_{s''}) \right)}{\sum_{a}  \exp\left( \tau^{-1}(r_{sa}+\g\sum_{s''} P_{ass''}v_{s''}) \right)} = \gamma.
\]
Hence $\varphi$ is a contraction map and converges to a fixed point, which is the solution to
\eqref{eq:primal} at a linear rate $O(\gamma^T)$, where $T$ is the number of iterations. After
obtaining the optimal value function $v$, the corresponding policy $\pi$ is given by
\cite{ying2020note}:
\eq{\label{eq:vtopi}
  \pi_{sa}=\frac{\exp\left(\tau^{-1}\left(r_{sa}+\gamma \sum_{s' \in \mathcal{S}} P_{ass'} v_{s'}\right) \right)}
     {\sum_a \exp \left(\tau^{-1}\left(r_{sa}+\gamma \sum_{s' \in \mathcal{S}} P_{ass'} v_{s'}\right) \right)}
     = \exp\left(\tau^{-1}\left(r_{sa} - \sum_{s'} (I-\gamma P_a)_{ss'} v_{s'}\right)\right).
}
As a result of the aforementioned double-sampling problem, the value-iteration algorithm based on \eqref{eq:primal} is mainly used in the model-based
setting, but due to the nice properties of $\varphi$, it appears as an important ingredient in
various other algorithms. For example, in \cite{asadi2017alternative} and
\cite{rawlik2013stochastic}, the authors use the function $\varphi$ as an alternative softmax
operator and form a $Q$-learning type algorithm, and in \cite{nachum2017trust}, the function $\varphi$
appears as a result of the inner optimization of an entropy regularized trust region-type formulation and is used to form the loss function. In \cite{dai2018sbeed}, the mean squared regularized Bellman error is employed to establish the optimization problem. 

An alternative way to solve a regularized Markov decision problem in the model-based setting is the dual formulation \cite{ying2020note}, in which one seeks a policy $\pi$ that solves the following optimization problem:
\eq{\label{eq:pg}
\max_\pi e^\T v_\pi := e^\T (I-\gamma P_\pi)^{-1}(r_\pi-\tau h_\pi),
}
where $e\succ0$ is a weight vector. By the existence and uniqueness of the optimal value function
and optimal policy and the optimality \eqref{eq:optimality}, it is clear that any choice of $e$
leads to the optimal policy and the optimal value function. A variety of policy gradient algorithms
can be used to solve the dual problem. Examples include
\cite{williams1992simple,sutton1999policy,kakade2001natural,schulman2015high,schulman2015trust,schulman2017proximal},
to mention only a few. Recently, \cite{li2021quasinewton} proposes a quasi-Newton policy gradient
algorithm, where an approximate Hessian of the objective function in \eqref{eq:pg} is used as a
preconditioner for the gradient, resulting in a quadratic convergence rate by better fitting the
problem geometry.

The word {\em primal-dual} also appears in other types of formulations where the dual variables do
not represent the policy. For example, in \cite{ding2020natural}, the authors apply the natural
policy gradient method to constrained MDPs (CMDPs), where the dual variables are the multipliers of
the constraints. Similarly, in \cite{chow2018lyapunov}, the dual variables come from the constraints
in CMDPs. In this paper, the Lyapunov method is used to give a theoretical analysis of the natural
gradient flow of the method we propose. The idea of Lyapunov methods has also been applied to
discrete time control problems \cite{kalman1960control,kalman1959control} and to discrete
Markovian systems \cite{meyn2012markov}. Recently it has also been used to address the safety
problem, where safety usually appears as additional constraints in the model
\cite{perkins2002lyapunov, chow2018lyapunov,berkenkamp2017safe}, and the Lyapunov function is
usually defined on the state space and is used explicitly in the policy iteration or in finding the
controller.

\subsection{Notations}  For a vector $x\in\mathbb{R}^d$, $\diag{(x)}$ denotes a diagonal matrix with size
$d\times d$ and the $k$-th diagonal element being $x_k$, $1\leq k\leq d$. For $u\in
\mathbb{R}^{|S||A|}$, we denote the $((s-1)|A|+a)$-th element as $u_{sa}$.  While $u_{s\cdot}$
denotes the vector in $\mathbb{R}^{|A|}$ with the $a$-th element being $u_{sa}$, $u_{\cdot a}$
denotes the vector in $\mathbb{R}^{|S|}$ with the $s$-th element being $u_{sa}$. The states of the
MDP are typically referred to with $s$, $s'$, and $s''$ while the actions are referred to by $a$ and
$a'$. The vector with length $d$ and all elements equal to $1$ is denoted by $\mathbf{1}_{d}$, and
the subscript $d$ is often omitted when there is no ambiguity. The $d$-by-$d$ identity matrix is
denoted by $I_{d}$, again with the subscript $d$ often omitted when there is no ambiguity. For a
matrix $B$, $B^\Hlg$ denotes its Hermitian transpose. If a scalar function is applied to a vector,
then the result is defined element-wise unless otherwise specified, e.g., for $x\in \mathbb{R}^d$,
$\exp(x)\in \mathbb{R}^d$ with $\exp(x)_k = \exp(x_k)$ for $1\leq k\leq d$.

\subsection{Contents}  
The rest of the paper is organized as follows. \Cref{sec:strict} derives the quadratically
convexified primal-dual formulation, proves its equivalence with \eqref{eq:linearpd}, and shows that
the vanilla NGAD of the new formulation converges linearly using a Lyapunov function method.
\Cref{sec:metric} introduces an interpolating metric by leveraging the flexibility of
the underlying metric described by the block diagonal part of the Hessian. The convergence rate of the
INGAD based on this new interpolating metric is significantly improved. We also provide a Lyapunov-style proof for global convergence and an analysis of the exponential convergence rate in the
last-iterate sense. Finally, \cref{sec:numerical} demonstrates the numerical performance of
these proposed natural gradient methods.

\section{Quadratically convexified primal-dual formulation}\label{sec:strict}

\subsection{Formulation}\label{sec:equivalence}

In what follows, we use $E_0(v,u)$ to denote the objective of the standard entropy regularized primal-dual formulation
\eq{\label{eq:pdstandard}
\min_v \max_u E_0(v, u) := \sum_{s} e_{s} v_{s} + \sum_{sa} u_{sa} (r_{sa}-(K_av)_s)             - \tau \sum_{sa} u_{sa} \log\frac{u_{sa}}{\bu_s}.
}
Since it is linear in $v$ and linear along the radial direction of $u$, first-order optimization methods
typically experience slow convergence.
To address the issue in the $v$ variable, we propose a quadratically convexified primal-dual formulation:
\eq{\label{eq:strict}
\min_v \max_u E(v, u) := \frac{\alpha}{2}\sum_{s} v_{s}^2 + \sum_{sa} u_{sa} (r_{sa}-(K_av)_s) - \tau \sum_{sa} u_{sa} \log\frac{u_{sa}}{\bu_s}.
}
Though these two formulations look quite different, they are indeed equivalent when $r_{sa}>0$ in the following sense.
\begin{itemize}
\item They share the same optimal value function $v^*$.
\item The optimal dual variable $u^*$ differs only by an $s$-dependent scaling factor. This implies
  that the optimal policy $\pi^*_{sa} \equiv u^*_{sa}/\bu^*_s$ are the same.
\end{itemize}

One geometric way to see this equivalence is to go through the associated primal formulations
\eq{\label{eq:primalopt}
\min_{v} e^{\T} v, \text { s.t. } \forall s, v_{s} \geq \tau\log \left(\sum_{a \in A} \exp \left(\frac{r_{sa}+\gamma \sum_{s'} P_{ass'} v_{s'}}{\tau}\right)\right),
}
and
\eq{\label{eq:primaloptquad}
\min_{v} \frac{\alpha}{2} \|v\|^2, \text { s.t. } \forall s, v_{s} \geq \tau\log \left(\sum_{a \in A} \exp \left(\frac{r_{sa}+\gamma \sum_{s'} P_{ass'} v_{s'}}{\tau}\right)\right).
}
\cref{fig:equivalence} illustrates the primal formulations of a randomly
generated MDP with $|S|=|A|=2$, where the yellow region
represents the feasible set and the red dot represents the optimal value $v^*$. Due to the key
assumption $r_{sa}\geq0$, the feasible set lies in the first quadrant. From the contour
plots of the objective function $e^{\T}v$ and $\|v\|^2$ shown by the dotted curves, it is
clear that both of them are minimized at $v^*$ when constrained to the feasible set.

\begin{figure}[ht!]
\centering
  \includegraphics[scale=0.5]{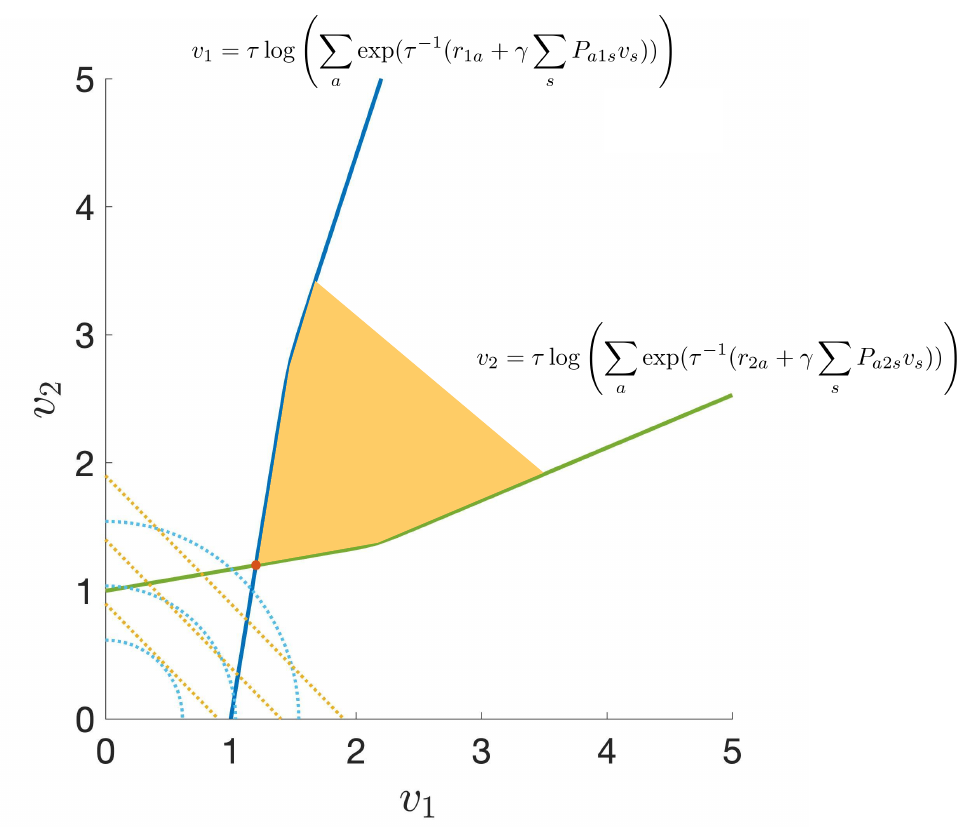}
  \caption{This plot heuristically demonstrates the correctness of the quadratically convexified
    primal-dual formulation on a randomly generated MDP with $|S|=|A|=2$. The yellow region represents the feasible set of the primal problem \eqref{eq:primalopt}, whose boundary corresponds to the solution to equation \eqref{eq:primal} and is shown by the blue and green curves. The red dot denotes the optimal value $v^*$. The cyan and orange dotted curves are contour lines of $\|v\|^2$ and $e^{\T} v$, respectively. It can be seen from this plot that the solution to the quadratically convexified formulation \eqref{eq:primaloptquad} is also $v^*$.  }
  \label{fig:equivalence}
\end{figure}

The following theorem states this equivalence formally, with its proof given in \cref{sec:proofs}.
\begin{theorem}
  \label{thm:equi}
  For an infinite-horizon discounted MDP with finite state space $S$, finite action space $A$ and nonnegative reward $r$, we have the following properties:
  \begin{enumerate}[(a)]
  \item There is a unique solution $(v^*, u^\circ)$ to the primal-dual problem: 
    \[
    \min_v \max_u E_0(v, u) = \sum_s e_s v_s + \sum_{sa} u_{sa} \left(r_{sa}- \sum_{s'} K_{ass'} v_{s'}\right) - \tau \sum_{sa} u_{sa} \log\frac{u_{sa}}{\bu_s},
    \]
    where $v^*$ is the optimal value function defined by \eqref{eq:optimality} and $\frac{u_{sa}^\circ}{\bu_s^\circ}$ gives the optimal policy $\pi_{sa}^*$.
  \item There is a unique solution $(v^*, u^*)$ to the quadratically convexified problem:
    \[
    \min_v \max_u E(v, u) = \frac{\alpha}{2}\sum_{s} v_{s}^2 + \sum_{sa} u_{sa} \left(r_{sa}- \sum_{s'} K_{ass'} v_{s'}\right) - \tau \sum_{sa} u_{sa} \log\frac{u_{sa}}{\bu_s},
    \]
    where $v^*$ is the optimal value function, and $\frac{u_{sa}^*}{\bu_s^*}$ coincides with the optimal policy $\pi_{sa}^*$.
  \end{enumerate}
\end{theorem}

\begin{remark}With the same method as the one used for the proof of \Cref{thm:equi}, one can show that the conclusions of \Cref{thm:equi} still hold if the term $\frac{\alpha}{2}\sum_s v_s^2$ in the formulation \eqref{eq:strict} is replaced with a strictly increasing convex function of $v$. The intuition provided in \Cref{fig:equivalence} also applies. 
\end{remark}

\subsection{Natural gradient ascent descent}\label{sec:ANG}

As mentioned earlier, the gradient-based methods for the primal-dual formulation
\eqref{eq:pdstandard} suffer from slow convergence, partly due to the linearity of $E_0(v, u)$ in
$v$.  Since the quadratically convexified scheme \eqref{eq:strict} gives the same value function
$v^*$ and policy $\pi^*$ as the original primal-dual problem \eqref{eq:pdstandard}, we work instead
with \eqref{eq:strict} and propose an NGAD algorithm.

The first-order derivatives of the new objective function $E(v, u)$ are
\eq{\label{eq:gradient}
  \begin{aligned}
    \frac{\p E}{\p v_{s'}} &= \alpha v_{s'} - \sum_{sa} K_{ass'} u_{sa}, \quad s' \in S, \\
    \frac{\p E}{\p u_{sa}} &= \left(r_{sa}-\sum_{s'} K_{ass'} v_{s'} \right) - \tau \log\frac{u_{sa}}{\bu_s}, \quad (s, a) \in S\times A.
  \end{aligned}
}
The diagonal blocks of the second-order derivatives  $\frac{\p^2 E}{\p v^2}$ and $\frac{\p^2 E}{\p u^2}$ are
\eq{\label{eq:hessian}
\begin{aligned}
  \frac{\p^2 E}{\p v_s \p v_{s'}} &= \alpha \delta_{ss'}, \quad (s, s')\in S\times S\\
  \frac{\p^2 E}{\p u_{sa} \p u_{s'a'}} &= -\tau \delta_{ss'} \left(\frac{\delta_{aa'}}{u_{sa}} - \frac{1}{\bu_s} \right), \quad (s, s', a, a')\in S^2\times A^2.
\end{aligned}
}
Of the two diagonal blocks above, $\frac{\p^2 E}{\p v^2}$ is easy to invert since it is diagonal
with positive diagonal entries, whereas $\frac{\p^2 E}{\p u^2}$ is the sum of a diagonal part and a
low-rank part. In the natural gradient dynamics below, we only keep the first part of $\frac{\p^2 E}{\p u^2}$,
namely $-\tau \delta_{ss'}\delta_{aa'}/u_{sa}$ (or more compactly $-\tau\diag{(1/u)}$ in the matrix form).
The resulting NGAD flow is:
\[
\begin{aligned}
  \dt{v_{s'}} & = - \frac{1}{\alpha} \left(\alpha v_{s'} - \sum_{sa} K_{ass'} u_{sa}\right), \quad s'\in S, \\
  \dt{u_{sa}} & = -\frac{1}{\tau} u_{sa} \left(\tau \log \frac{u_{sa}}{\bu_{s}} - \left(r_{sa}-\sum_{s'} K_{ass'} v_{s'}\right) \right), \quad (s, a)\in S\times A,
\end{aligned}
\]
or equivalently,  
\eq{\label{eq:ANGflow1}
\begin{aligned}
  \dt{v_{s'}} & = - \left(v_{s'} - \frac{1}{\alpha} \sum_{sa} K_{ass'} u_{sa} \right), \quad s'\in S,  \\
  \dt{u_{sa}} & = - u_{sa} \left(\log \frac{u_{sa}}{\bu_{s}} - \frac{1}{\tau} \left(r_{sa}-\sum_{s'} K_{ass'} v_{s'}\right) \right), \quad (s, a)\in S\times A.
\end{aligned}
}

To analyze its convergence, we start by identifying a Lyapunov function of this dynamics. By ~\Cref{thm:equi} there is a unique solution $(v^*, u^*)$ to problem \eqref{eq:strict}.  Based
on the solution $(v^*, u^*)$, define
\begin{equation}
  \label{eq:Lyap1}
  L(v,u) = \frac{\alpha}{2} \sum_{s\in S} |v_{s}-v^*_{s}|^2 + \tau \sum_{s\in S,a\in A} \left(u^*_{sa}\log
  \frac{u^*_{sa}}{u_{sa}} + u_{sa}-u^*_{sa} \right).
\end{equation}
The following lemma summarizes some key properties of $L(v,u)$.
\begin{lemma}
  \label{lemma:L}
  $L(v,u)\geq0$ is strictly convex, and the unique minimum is $(v^*,u^*)$, which satisfies
  $L(v^*,u^*)=0$. In addition, any sublevel set of $L$ is bounded.
\end{lemma}
The next lemma states that $L(v, u)$ is a Lyapunov function of \eqref{eq:ANGflow1}.
\begin{lemma}
\label{lemma:lyap}
  $L(v,u)$ is a Lyapunov function for the dynamics \eqref{eq:ANGflow1}, i.e., $\dt{L}\leq0$ when
$\dt{v}$ and $\dt{u}$ are defined in \eqref{eq:ANGflow1}, and the only trajectory of the dynamics
\eqref{eq:ANGflow1} satisfying $\dt{L}=0$ is $(v, u) = (v^*, u^*)$.
\end{lemma}
The proofs of these two lemmas are given in \cref{sec:proofs}.

\begin{theorem}
\label{thm:conv1}
The dynamics of \eqref{eq:ANGflow1} converges globally to $(v^*,u^*)$.
\end{theorem}
\begin{proof}
By \Cref{lemma:L}, \Cref{lemma:lyap} and the Barbashin-Krasovskii-LaSalle theorem
\cite{haddad2011nonlinear}, the dynamics of \eqref{eq:ANGflow1} is globally asymptotically stable,
which means the NGAD dynamics converges globally to $(v^*, u^*)$.
\end{proof}

To show the exponential convergence of \eqref{eq:ANGflow1}, we follow Lyapunov's indirect method,
i.e., analyzing the linearization of \eqref{eq:ANGflow1} at $(v^*, u^*)$ and demonstrating that the
real part of the eigenvalues of the corresponding matrix is negative. This result is the content of \Cref{thm:linrate}, with the proof given in \cref{sec:proofs}.

\begin{theorem}
\label{thm:linrate}
The dynamics of \eqref{eq:ANGflow1} converges at rate $O(e^{-ct})$ to $(v^*,u^*)$ for some $c>0$.
\end{theorem}

Below we discuss the implementation of \eqref{eq:ANGflow1}. By introducing $u_{sa} = \exp(\theta_{sa})$, \eqref{eq:ANGflow1} can be rewritten as
\eq{\label{eq:ANGflow1_p}
\begin{aligned}
  \dt{v_{s'}} & = - \left(v_{s'} - \frac{1}{\alpha} \sum_{sa} K_{ass'} \exp(\theta_{sa}) \right), \quad s'\in S,  \\
  \dt{\theta_{sa}} & = - \left(\theta_{sa}-\log\left(\sum_a \exp(\theta_{sa})\right) - \frac{1}{\tau} \left(r_{sa}-\sum_{s'} K_{ass'} v_{s'}\right) \right), \quad (s, a)\in S\times A.
\end{aligned}
}
With a learning rate $\eta>0$, this leads to the update rule
\eq{\label{eq:ANGflow1_d}
\begin{aligned}
  v_{s'} & \leftarrow (1-\eta)v_{s'} + \frac{\eta}{\alpha} \sum_{sa} K_{ass'} \exp(\theta_{sa}),
  \quad s'\in S, \\ \theta_{sa} & \leftarrow (1-\eta)\theta_{sa}+\eta\log\left(\sum_a
  \exp(\theta_{sa})\right) + \frac{\eta}{\tau} \left(r_{sa}-\sum_{s'} K_{ass'} v_{s'}\right), \quad
  (s, a)\in S\times A.
\end{aligned}
}

The details of the algorithm are summarized in \Cref{alg:ANG1}. 
\begin{algorithm}[htbp]
    \caption{Standard NGAD for quadratically convexified formulation} 
    \label{alg:ANG1}
    \begin{algorithmic}[1]
        \REQUIRE the MDP model $\mathcal{M} = (S, A, P, r, \gamma)$, initialization
        $(v_{\text{init}}, \theta_{\text{init}})$, convergence threshold $\epsilon_{\text{tol}}$,
        coefficient $\alpha>0$ for the quadratic term in \eqref{eq:strict}, regularization
        coefficient $\tau$, learning rate $\eta$.

        \STATE Initialize the value and parameters $v = v_{\text{init}}$, $\theta = \theta_{\text{init}}$.

        \STATE Calculate $u_{sa} = \exp(\theta_{sa})$, $(s, a)\in S\times A$. 

        \STATE Set $q = 1 + \epsilon_{\text{tol}}$.

        \WHILE {$q > \epsilon_{\text{tol}}$}
        
        \STATE Calculate $(v_{\text{new}})_{s'} = (1-\eta)v_{s'} + \frac{\eta}{\alpha} \sum_{sa} K_{ass'} u_{sa}, \quad s'\in S$. 

        \STATE Update $\theta$ by \[\theta_{sa}\leftarrow (1-\eta)\theta_{sa}+\eta\log\sum_a u_{sa}
        + \frac{\eta}{\tau} \left(r_{sa}-\sum_{s'} K_{ass'} (v_{\text{new}})_{s'}\right), ~(s, a)\in
        S\times A.\]

        \STATE Calculate $(u_{\text{new}})_{sa} = \exp(\theta_{sa})$, $(s, a)\in S\times A$. 

        \STATE Calculate $q = \max\{\|v_{\text{new}}-v\| / \|v\|, \|u_{\text{new}}-u\| / \|u\|\}$.
        \STATE Update $(v, u)$ by $v\leftarrow v_{\text{new}}$, $u\leftarrow u_{\text{new}}$. 
        \ENDWHILE
    \end{algorithmic}
    \vspace{0.5em}
\end{algorithm}

\section{Interpolating natural gradient method}\label{sec:metric}

In \cref{sec:ANG}, NGAD is introduced using the diagonal part of
$\frac{\p^2 E}{\p u^2}$. A natural question is whether the whole matrix
$\frac{\p^2 E}{\p u^2}$ can be used.  Under the matrix notation,
$\frac{\p^2 E}{\p u^2}$ in \eqref{eq:hessian} takes the form
\eq{\label{eq:uhess} \frac{\p^2
    E}{\p u^2}=
\begin{bmatrix}
H_1&&\\
&\ddots&\\
&&H_{|S|}
\end{bmatrix},\quad
H_s = \diag{\left((u_{s\cdot})^{-1}\right)}-\frac{1}{\bu_s}\mathbf{1}_{|A|}\mathbf{1}_{|A|}^\T, \quad s\in S.
}
Since the Hessian matrix describes the local geometry of the problem, the standard NGAD in
Section~\cref{sec:ANG} can be viewed as approximating the Hessian diagonally 
\[
\frac{\p^2 E}{\p u^2}\approx
\begin{bmatrix}
\diag{\left((u_{1\cdot})^{-1}\right)}&&\\
&\ddots&\\
&&\diag{\left((u_{|S|\cdot})^{-1}\right)}
\end{bmatrix}
\]
and using its inverse 
\[
\begin{bmatrix}
\diag(u_{1\cdot})&&\\
&\ddots&\\
&& \diag(u_{|S|\cdot})
\end{bmatrix}
\equiv
\begin{bmatrix}
\bu_1 (\diag(\pi_{1\cdot}))&&\\
&\ddots&\\
&&\bu_{|S|}  (\diag(\pi_{|S|\cdot}))
\end{bmatrix}
\]
to precondition the gradient.  
However, $H_s$ is in
fact singular with $\operatorname{Null}(H_s) = \operatorname{Span}(u_{s\cdot})$ and its
pseudoinverse reads
\[
\begin{bmatrix}
\bu_1 (\diag(\pi_{1\cdot}) - \pi_{1\cdot}\pi_{1\cdot}^\T)&&\\
&\ddots&\\
&&\bu_{|S|}  (\diag(\pi_{|S|\cdot}) - \pi_{|S|\cdot}\pi_{|S|\cdot}^\T)
\end{bmatrix}.
\]
If we had constructed the natural gradient method with this pseudoinverse, the component in the $\mathbf{1}_{|A|}$
direction would not have been updated in the dynamics.


The key idea is that one can interpolate between these two extreme cases, i.e., we propose
to use
\begin{equation}\label{eq:newmetric}
  \begin{bmatrix}
    \bu_1 (\diag(\pi_{1\cdot}) - c \pi_{1\cdot}\pi_{1\cdot}^\T)&&\\
    &\ddots&\\
    &&\bu_{|S|}  (\diag(\pi_{|S|\cdot}) - c \pi_{|S|\cdot}\pi_{|S|\cdot}^\T)
  \end{bmatrix}
\end{equation}
for $0 < c < 1$ to precondition the gradient.

Under this interpolating metric \eqref{eq:newmetric}, the new {\em interpolating} NGAD (INGAD) is given
by
\eq{\label{eq:ANGflow2}
\begin{aligned}
  \dt{v_{s'}} & = - \left(v_{s'} - \frac{1}{\alpha} \sum_{sa} K_{ass'} u_{sa}\right), \quad s'\in S,\\
  \dt{u_{s\cdot}} & = - \bu_s  \left(\diag(\pi_{s\cdot})-c\pi_{s\cdot}\pi_{s\cdot}^\T\right)  \left(\log \frac{u_{s\cdot}}{\bu_{s}} - \frac{1}{\tau}\left(r_{s\cdot}-\sum_{s'} K_{\cdot ss'} v_{s'}\right) \right), \quad s\in S,
\end{aligned}
}
where $u_{s\cdot}\in\mathbb{R}^{|A|}$. When $c=0$, this dynamics reduces to \eqref{eq:ANGflow1}. 

A Lyapunov function of this dynamics can also be identified. Using the unique solution $(v^*,u^*)$ to \eqref{eq:strict}, we define 
\eq{\label{eq:lyap2} L_c(v,u) =
  \frac{\alpha}{2}\sum_{s} |v_{s}-v^*_{s }|^2 + \tau \left( \sum_{sa} \left(u^*_{sa}\log
  \frac{u^*_{sa}}{u_{sa}} + u_{sa}-u^*_{sa}\right)
  + \frac{c}{1-c} \sum_s \left(\bu^*_{s}\log \frac{\bu^*_{s}}{\bu_{s}} +
  \bu_{s}-\bu^*_{s}\right) \right),
}
where the subscript $c$ denotes the hyperparameter in the function. Some key properties
of $L_c(v,u)$ are summarized in the following lemma.
\begin{lemma}
  \label{lemma:Lc}
  $L_c(v,u)$ is convex and the unique minimum is $L_c(v^*,u^*) = 0$. The sublevel sets of $L_c$ are
  bounded.
\end{lemma}
The next lemma states that $L_c(v,u)$ is a Lyapunov function for \eqref{eq:ANGflow2}.
\begin{lemma}
  \label{lemma:lyap2}
  $L_c(v,u)$ is a Lyapunov function for the dynamics \eqref{eq:ANGflow2}, i.e., $\dt{L_c}\leq0$ when
  $\dt{v}$ and $\dt{u}$ are defined by \eqref{eq:ANGflow2}, and the only trajectory of the dynamics
  \eqref{eq:ANGflow2} satisfying $\dt{L_c}=0$ is $(v, u) = (v^*, u^*)$.
\end{lemma}
The proofs of these two lemmas can be found again in \cref{sec:proofs}.

\begin{theorem}
\label{thm:conv2}
The dynamics of \eqref{eq:ANGflow2} converges globally to $(v^*,u^*)$.
\end{theorem}
\begin{proof}
Similar to \Cref{thm:conv1}, by \Cref{lemma:Lc}, \Cref{lemma:lyap2} and the
Barbashin-Krasovskii-LaSalle theorem \cite{haddad2011nonlinear}, the dynamics of \eqref{eq:ANGflow2}
is globally asymptotically stable and hence converges globally to $(v^*, u^*)$.
\end{proof}

The local exponential convergence of \eqref{eq:ANGflow2} can also be shown with Lyapunov's indirect
method. This result is stated in \Cref{thm:linrate2}.
\begin{theorem}
\label{thm:linrate2}
  The dynamics of \eqref{eq:ANGflow2} converges at rate $O(e^{-ct})$ to $(v^*,u^*)$ for some $c>0$.
\end{theorem}

Finally, we discuss the implementation of \eqref{eq:ANGflow1}. By letting $u_{sa} = \exp(\theta_{sa})$,
\eqref{eq:ANGflow2} can be written as
\eq{\label{eq:ANGflow2_p}
\begin{aligned}
  \dt{v_{s'}} & = - \left(v_{s'} - \frac{1}{\alpha} \sum_{sa} K_{ass'} \exp(\theta_{sa})\right), \quad s'\in S,\\
  \dt{\theta_{s\cdot}} & = - \left(I-\frac{c\mathbf{1}\exp(\theta_{s\cdot})^\T}{\mathbf{1}^\T\exp(\theta_{s\cdot})}\right)  
  \left(\theta_{s\cdot}-\log\sum_a\exp(\theta_{sa})\mathbf{1} - \frac{1}{\tau}\left(r_{s\cdot}-\sum_{s'} K_{\cdot ss'} v_{s'}\right) \right), ~ s\in S.
\end{aligned}
}
With a learning rate $\eta>0$, this becomes
\eq{\label{eq:ANGflow2_d}
\begin{aligned}
  v_{s'} & \leftarrow (1-\eta)v_{s'} + \frac{\eta}{\alpha} \sum_{sa} K_{ass'} \exp(\theta_{sa}), \quad s'\in S,  \\
  \theta_{s\cdot} & \leftarrow \theta_{s\cdot}-\eta\left(I-\frac{c\mathbf{1}\exp(\theta_{s\cdot})^\T}{\mathbf{1}^\T\exp(\theta_{s\cdot})}\right)  
  \left(\theta_{s\cdot}-\log\sum_a\exp(\theta_{sa})\mathbf{1} - \frac{1}{\tau}\left(r_{s\cdot}-\sum_{s'} K_{\cdot ss'} v_{s'}\right) \right), ~s\in S.
\end{aligned}
}
The details of the algorithm can be found in \Cref{alg:ANG2} below. 
\begin{algorithm}[htbp]
    \caption{INGAD for quadratically convexified formulation} 
    \label{alg:ANG2}
    \begin{algorithmic}[1]
        \REQUIRE the MDP model $\mathcal{M} = (S, A, P, r, \gamma)$, initialization $(v_{\text{init}}, \theta_{\text{init}})$, convergence threshold $\epsilon_{\text{tol}}$, coefficient $\alpha>0$ for the quadratic term in \eqref{eq:strict},
        regularization coefficient $\tau$, metric coefficient $0\leq c<1$, learning rate $\eta$.
        \STATE Initialize the value and parameters $v = v_{\text{init}}$, $\theta = \theta_{\text{init}}$.
        \STATE Calculate $u_{sa} = \exp(\theta_{sa})$, $(s, a)\in S\times A$. 
        \STATE Set $q = 1 + \epsilon_{\text{tol}}$.
        \WHILE {$q > \epsilon_{\text{tol}}$}
        
        \STATE Calculate $(v_{\text{new}})_{s'} = (1-\eta)v_{s'} + \frac{\eta}{\alpha} \sum_{sa} K_{ass'} u_{sa}, \quad s'\in S$. 
        \STATE Update $\theta$ by \[\theta_{s\cdot}\leftarrow\theta_{s\cdot}-\eta\left(I-\frac{c\mathbf{1}u_{s\cdot}^\T}{\mathbf{1}^\T u_{s\cdot}}\right)
        \left(\theta_{s\cdot}-\left(\log\sum_au_{sa}\right)\mathbf{1} - \frac{1}{\tau}\left(r_{s\cdot}-\sum_{s'} K_{\cdot ss'} (v_{\text{new}})_{s'}\right) \right), ~s\in S.\]

        \STATE Calculate $(u_{\text{new}})_{sa} = \exp(\theta_{sa})$, $(s, a)\in S\times A$. 

        \STATE Calculate $q = \max\{\|v_{\text{new}}-v\| / \|v\|, \|u_{\text{new}}-u\| / \|u\|\}$.
        \STATE Update $(v, u)$ by $v\leftarrow v_{\text{new}}$, $u\leftarrow u_{\text{new}}$. 
        \ENDWHILE
    \end{algorithmic}
    \vspace{0.5em}
\end{algorithm}

\section{Numerical results}\label{sec:numerical}
In this section, we examine the performance of \Cref{alg:ANG1} and \Cref{alg:ANG2} with several different examples. \Cref{sec:exp1} compares \Cref{alg:ANG1} and \Cref{alg:ANG2} in a complete-information case where the transition probabilities and the rewards are known exactly. A comparison with an existing method in \cite{zhan2021policy} is showcased in this setting as well. The sample-based setting is investigated in \cref{sec:exp3}, where we give an adapted version of INGAD with sample access, and test its performance on two different MDPs.

\subsection{Experiments with complete information}\label{sec:exp1}
Here we test the numerical performance of the standard natural gradient in \Cref{alg:ANG1} and the interpolating natural gradient in \Cref{alg:ANG2} in a complete information situation. The MDP used is from \cite{zhan2021policy}, where $|S| = 200$, $|A| = 50$, and the transition probabilities and rewards are randomly generated. More specifically, the transition probabilities are set as $P_{ass'} = 1/20$ for any $s'\in S_{sa}$, where $S_{sa}$ is a uniformly randomly chosen subset of $S$ such that $|S_{sa}| = 20$, and the reward $r_{sa} =U_{sa}U_s$ for $(s,
a)\in S\times A$, where $U_{sa}$ and $U_s$ are independently uniformly sampled from $[0,1]$.

A comparison of \Cref{alg:ANG1} and \Cref{alg:ANG2} is carried out using the same discount rate $\g = 0.99$ and hyperparameters $(\epsilon_{\text{tol}}, \alpha, \tau) = (1\times10^{-5}, 0.1, 0.01)$. Since both algorithms are explicit discretizations of the corresponding flow, a sufficiently small learning
rate is needed to ensure convergence. In the tests, the learning rates are set as
$\eta=3\times10^{-4}$ for \Cref{alg:ANG1} and $\eta=8\times10^{-3}$ for \Cref{alg:ANG2}, which are both manually tuned to be close to the largest learning rates such that convergence is achieved. For \Cref{alg:ANG2}, we set $c=0.98$. 

\begin{figure}[ht!]
  \centering
  \subfigure[Error of the value function and policy using \Cref{alg:ANG1}.\label{fig:vp1}]{
    \includegraphics[width=0.34\textwidth]{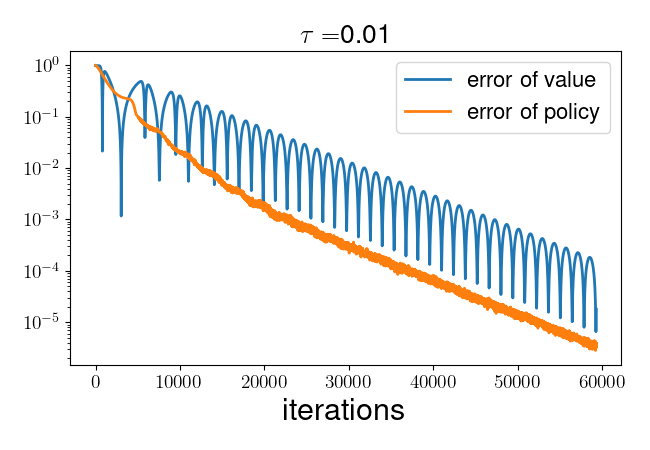} }
  \hspace{0em}
  \subfigure[The Lyapunov function. \label{fig:lyap1}]{
    \includegraphics[width=0.34\textwidth]{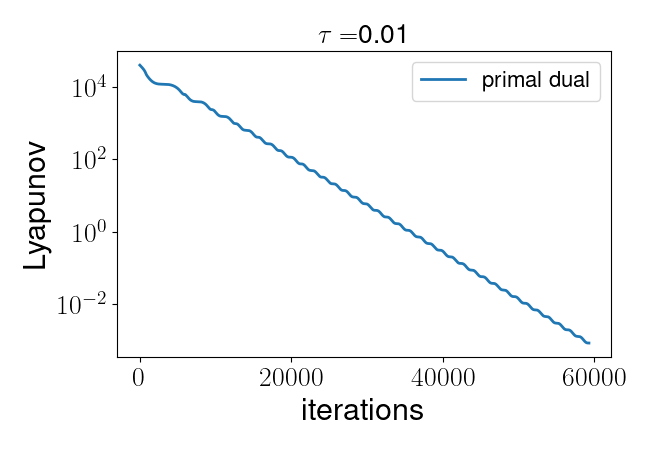} }
  \vspace{-0.5em}
  \subfigure[Error of the value function and policy using \Cref{alg:ANG2}.\label{fig:vp2}]{
    \includegraphics[width=0.34\textwidth]{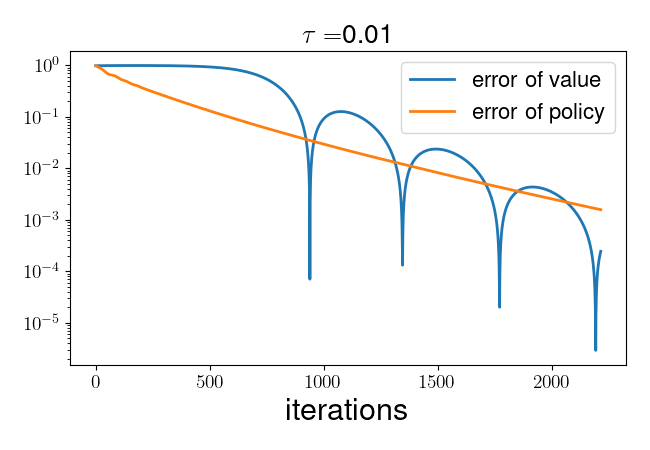} }
  \hspace{0em}
  \subfigure[The Lyapunov function. \label{fig:lyap2}]{
    \includegraphics[width=0.34\textwidth]{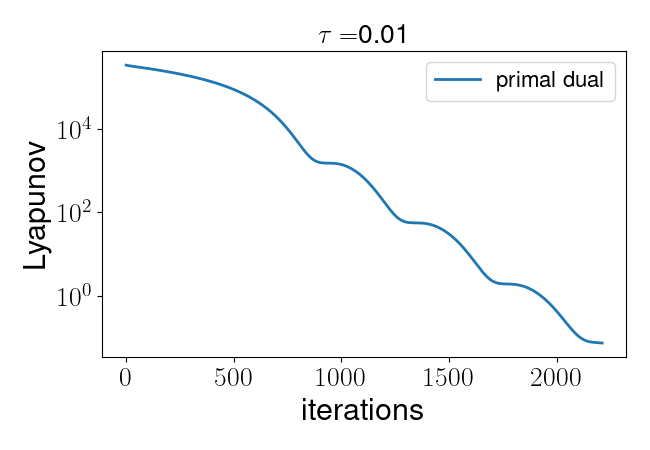} }
  \caption{Comparison of \Cref{alg:ANG1} and \Cref{alg:ANG2}. (a): Convergence of
    the value and policy during training of \Cref{alg:ANG1}; (b): Lyapunov function
    \eqref{eq:Lyap1}; (c): Convergence of the value and policy during training of
    \Cref{alg:ANG2}; (d): Lyapunov function \eqref{eq:lyap2}.  Blue curves in (a) and (c):
    The convergence of $\|\pi-\pi^*\|_F/\|\pi^*\|_F$ in the training process. Orange curves in (a)
    and (c): The convergence of $\|v-v^*\|_2/\|v^*\|_2$ in the training process. A logarithmic scale
    is used for all vertical axes.}
  \label{fig:comparealg}
\end{figure}
As a result, \Cref{alg:ANG1} takes $59296$ iterations to converge while \Cref{alg:ANG2} takes $2213$ iterations, demonstrating that the interpolating metric introduced in \cref{sec:metric} gives rise to an acceleration of more than $1$ magnitude. Plotted in \Cref{fig:vp1} and \Cref{fig:vp2} are the errors of the value and policy with respect to the ground truth in the training process, which verifies that \Cref{alg:ANG2} achieves the same precision more than a magnitude faster than \Cref{alg:ANG1}. Moreover, it can be observed from \Cref{fig:lyap1} and \Cref{fig:lyap2} that the Lyapunov function decreases monotonically in both cases, confirming the theoretical analyses in \cref{sec:strict} and \cref{sec:metric}.

\textbf{Comparison with PMD \cite{zhan2021policy}.}
Next, we compare the performance of \Cref{alg:ANG2} (INGAD) with an existing method, namely the policy mirror descent (PMD) method used in \cite{zhan2021policy}. The underlying MDP of the problem is the same as in
\cref{sec:exp1}. 
For the hyperparameters of INGAD, we take $(N_{\text{iter}}, \alpha, c)=(2000, 0.1, 0.98)$. In order to make a fair comparison, the learning rate is set as $\eta = 8\times 10^{-3}$, and the regularization coefficient is set as $\tau=0.01$ for both methods. For the PMD method, we take the first $20000$ iterations. 

\begin{figure}[ht]
  \centering
  \subfigure[Comparison of the error curves of the value function.\label{fig:vpgauss1}]{
    \includegraphics[width=0.34\textwidth]{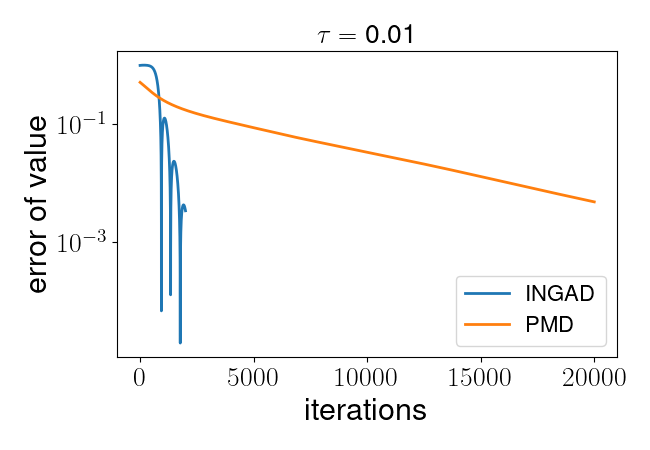} }
  \hspace{0em}
  \subfigure[Comparison of the error curves of the policy function.\label{fig:vpgauss2}]{
    \includegraphics[width=0.34\textwidth]{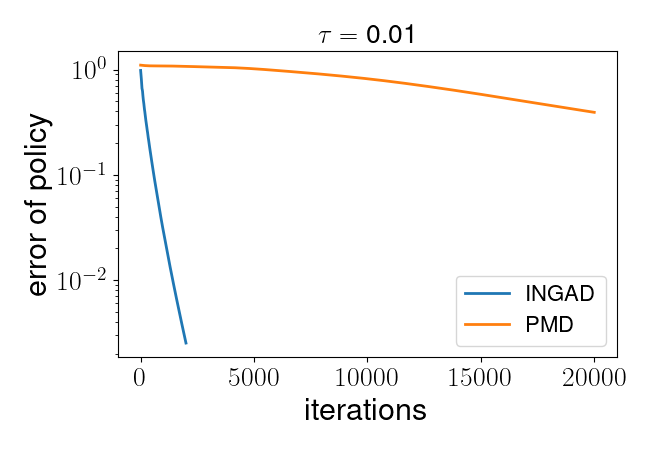} }
  \hspace{0em}
  \caption{Comparison of \Cref{alg:ANG2} with PMD \cite{zhan2021policy}. (a): Convergence
    of the value function; (b): Convergence of the policy. Blue
    curves: The convergence of $\|\pi-\pi^*\|_F/\|\pi^*\|_F$ in the training
    process. Orange curves: The convergence of $\|v-v^*\|_2/\|v^*\|_2$ in the
    training process. A logarithmic scale is used for all vertical axes.}
  \label{fig:randreward}
\end{figure}
It can be seen from \Cref{fig:randreward} that \Cref{alg:ANG2} admits a faster convergence than PMD. For both the value function and the policy, \Cref{alg:ANG2} achieves a higher precision in $2000$ iterations than PMD with $20000$ iterations. The final errors in the value function and policy are approximately $(0.0034, 0.0025)$ for INGAD and $(0.39, 0.0049)$ for PMD. 

\begin{algorithm}[ht!]
    \caption{INGAD for quadratically convexified formulation (sample version)} 
    \label{alg:ANG2_online}
    \begin{algorithmic}[1]
      \REQUIRE the discount rate $\gamma$, initialization $(v_{\text{init}}, \theta_{\text{init}})$,
      convergence threshold $\epsilon_{\text{tol}}$, maximum number of iterations $N_{\text{iter}}$, coefficient $\alpha>0$ for the quadratic term in \eqref{eq:strict},
      regularization coefficient $\tau$, metric coefficient $0\leq c<1$, the initial and the final learning rate $(\eta_{\text{init}}, \eta_{\text{end}})$.
      \STATE Initialize the value and parameters $v = v_{\text{init}}$, $\theta = \theta_{\text{init}}$.
      \STATE Calculate $u_{sa} = \exp(\theta_{sa})$, $(s, a)\in S\times A$. 
      \STATE Set $q = 1 + \epsilon_{\text{tol}}$ and $i=0$. 
      \STATE Initialize a buffer $\mathcal{B}$ with $N$ transition samples $(s, a, s', r)$.
      \WHILE {$q > \epsilon_{\text{tol}}$ and $i<N_{\text{iter}}$}
      \STATE Calculate $\eta_i = (1+i(N_{\text{iter}}\eta_{\text{end}})^{-1}(\eta_{\text{init}}-\eta_{\text{end}}))^{-1}\eta_{\text{init}}$
      \STATE Randomly sample a batch of samples from $\mathcal{B}$ with size $N_b$. 
      \STATE Estimate $\hat{K}^{(i)}$ from the samples. 
      \STATE Calculate $(v_{\text{new}})_{s'} = (1-\eta)v_{s'} + \frac{\eta}{\alpha} \sum_{sa} \hat{K}_{ass'}^{(i)} u_{sa}, \quad s'\in S$. 
      \STATE Update $\theta$ by \[\theta_{s\cdot}\leftarrow\theta_{s\cdot}-\eta\left(I-\frac{c\mathbf{1}u_{s\cdot}^\T}{\mathbf{1}^\T u_{s\cdot}}\right)
      \left(\theta_{s\cdot}-\left(\log\sum_au_{sa}\right)\mathbf{1} - \frac{1}{\tau}\left(r_{s\cdot}-\sum_{s'} \hat{K}_{\cdot ss'}^{(i)} (v_{\text{new}})_{s'}\right) \right), ~s\in S.\]
      
      \STATE Calculate $(u_{\text{new}})_{sa} = \exp((\theta_{\text{new}})_{sa})$, $(s, a)\in S\times A$. 
      
      \STATE Calculate $q = \max\{\|v_{\text{new}}-v\| / \|v\|, \|u_{\text{new}}-u\| / \|u\|\}$.
      \STATE Update $(v, u)$ by $v\leftarrow v_{\text{new}}$, $u\leftarrow u_{\text{new}}$. 
      \STATE $i\leftarrow i+1$.
      \ENDWHILE
    \end{algorithmic}
    \vspace{0.5em}
\end{algorithm}

\subsection{Experiments with random samples}\label{sec:exp3}
Finally, we test the INGAD algorithm on the case where the transition probabilities are unknown. In each iteration, a size-$N_b$ batch of samples is used to estimate the transition probabilities and used for the INGAD update, as presented in \Cref{alg:ANG2_online}. In order to stabilize the training dynamics, we use a decaying learning rate starting with $\eta_{\text{init}}$ and ending with $\eta_{\text{end}}$. If $\eta_{\text{init}}=\eta_{\text{end}}$, then the algorithm reduces to the constant learning rate case. We first use the MDP introduced in \cref{sec:exp1}. 

In this experiment, we adopt $(\gamma,N_{\text{iter}}, N_b, \alpha,\tau,c)=(0.9,12000, 1\times10^5, 0.1,0.1,0.9)$ and $(\eta_{\text{init}},\eta_{\text{end}})=(0.001, 0.001)$. Altogether $1\times10^8$ samples are used in the training process.
\begin{figure}[ht]
  \centering
  \subfigure[Error of the value function and policy.\label{fig:vponline}]{
    \includegraphics[width=0.34\textwidth]{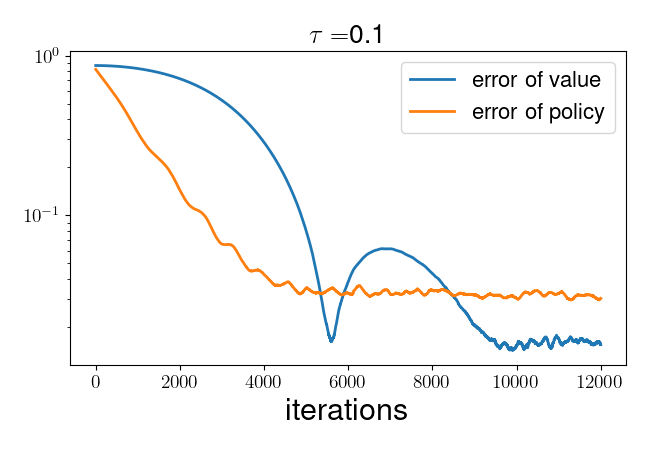} }
  \hspace{0em}
  \subfigure[The Lyapunov function. \label{fig:lyaponline}]{
    \includegraphics[width=0.34\textwidth]{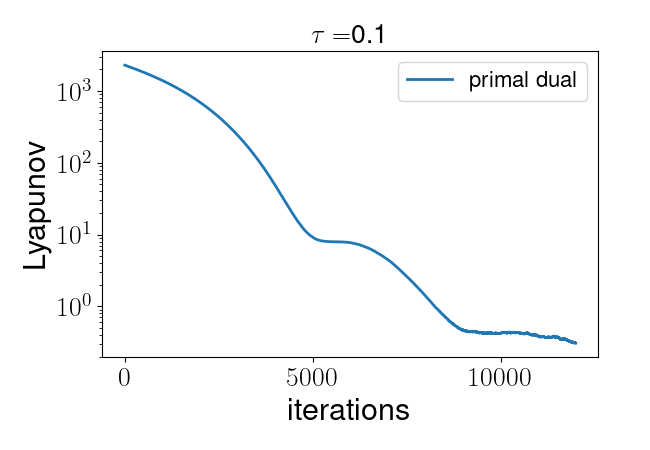} }
  \caption{Performance of \Cref{alg:ANG2_online} for the MDP problem described in \Cref{sec:exp1}. (a): Convergence of the value and policy
    during training of \Cref{alg:ANG2_online}; Blue curve: the convergence of
    $\|\pi-\pi^*\|_F/\|\pi^*\|_F$ in the training process; Orange curve: the convergence of
    $\|v-v^*\|_2/\|v^*\|_2$ in the training process. (b): Lyapunov function
    \eqref{eq:lyap2}. A logarithmic scale is used for all vertical axes.}
  \label{fig:online}
\end{figure}
It can be seen from \Cref{fig:online} that the approximate value function and policy given by 
\Cref{alg:ANG2_online} converge to the ground truth and oscillate around it at the final
stage. The final errors in the value function and policy are approximately $0.015$ and $0.030$,
respectively. It can also be seen from \Cref{fig:lyaponline} that the Lyapunov function mostly
decreases in the training process even though the transition probabilities used are just unbiased
estimators of the ground truth. 

\textbf{Experiment with the FrozenLake environment.}
In this part, the MDP we consider is from the FrozenLake environment (see \cite{brockman2016openai}). The environment describes the problem where the player aims to walk on a frozen lake from one corner to another without falling into the holes. In the example we use below, the map is an $8\times 8$ square grid with $10$ randomly generated holes. Therefore, the size of the state space is $64$, and there are $4$ actions, corresponding to the $4$ directions one can choose at each position. In order to model the low-friction property of ice, the transition is not deterministic. More specifically, the agent has a $1/3$ probability of moving in the intended direction or the two perpendicular directions. An illustration of the lake map is given in \Cref{fig:map}.

\begin{figure}[ht!]
\centering
  \includegraphics[scale=0.16]{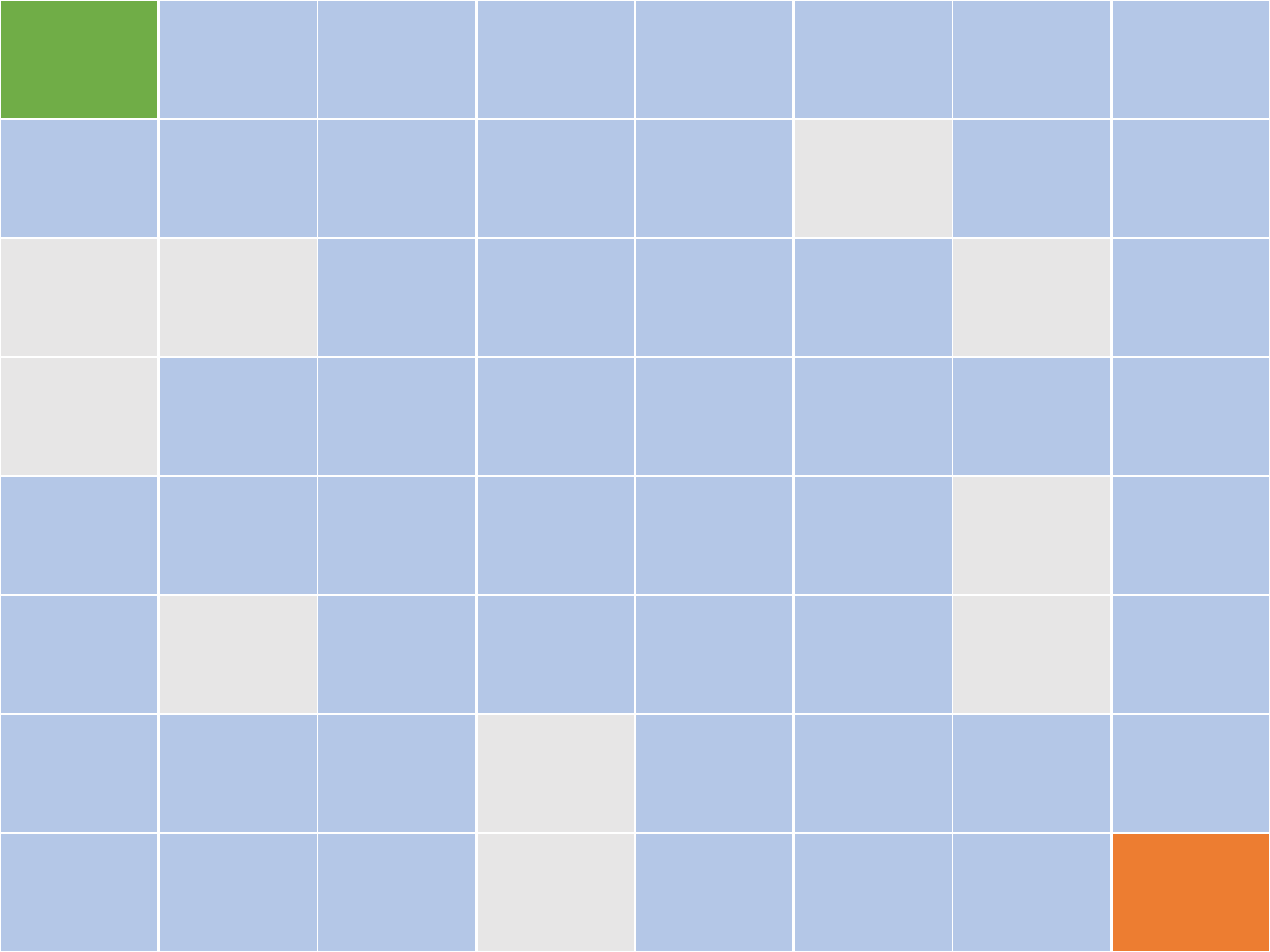}
  \caption{Map of the FrozenLake environment with size $8\times8$ and $10$ randomly generated holes. The green and the orange boxes represent the starting position and the target position, respectively. The blue area represents the positions with ice, while the grey spots indicate the positions of holes.}
  \label{fig:map}
\end{figure}

In the numerical experiment, we set
$(\gamma,N_{\text{iter}}, N_b,\alpha,\tau,c)=(0.9,80000,2000,0.1,0.1,0.9)$, and $(\eta_{\text{init}},\eta_{\text{end}})=(0.002,0.0002)$. The buffer size $N$ and
the batch-size $N_b$ are chosen as $2\times10^6$ and $2000$, respectively.
\begin{figure}[ht]
  \centering
  \subfigure[Error of the value function and policy.\label{fig:vponlinefrozen}]{
    \includegraphics[width=0.34\textwidth]{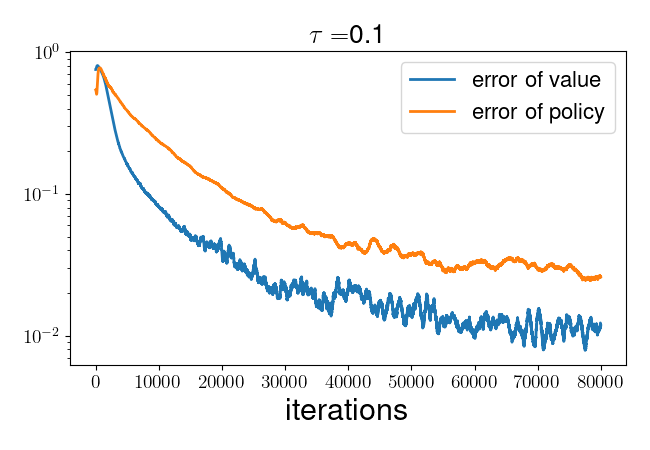} }
  \hspace{0em}
  \subfigure[The Lyapunov function. \label{fig:lyaponlinefrozen}]{
    \includegraphics[width=0.34\textwidth]{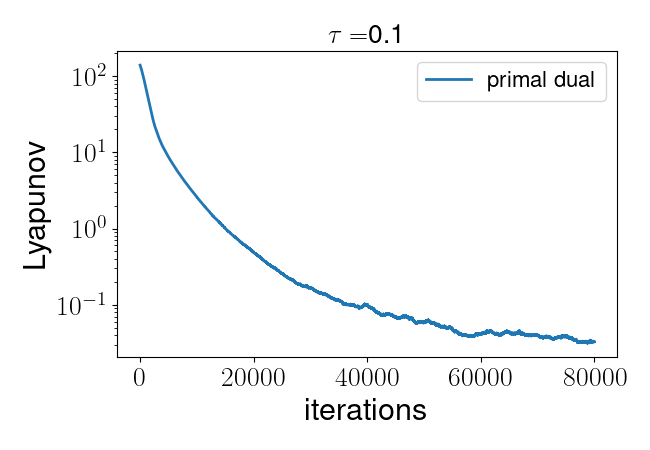} }
  \caption{Performance of \Cref{alg:ANG2_online} for the $8\times8$ FrozenLake problem. (a): Convergence of the value and policy during training of \Cref{alg:ANG2_online}; Blue curve: the convergence of $\|\pi-\pi^*\|_F/\|\pi^*\|_F$ in the training process; Orange curve: the convergence of $\|v-v^*\|_2/\|v^*\|_2$ in the training process. (b): Lyapunov function \eqref{eq:lyap2}. A logarithmic scale is used for all vertical axes.}
  \label{fig:onlinefrozen}
\end{figure}

Similar to the previous example, both the error of the value function and the error of the policy function reduce in the training process, indicating the effectiveness of \Cref{alg:ANG2_online} given sample access to the MDP. The oscillations represent the randomness in the samples gathered in each batch. The final errors for the value and policy are $0.012$ and $0.026$, respectively. The Lyapunov function also shows a clear decreasing trend along the training process. 

\section{Conclusion and discussion}\label{sec:conclusion}
In this paper, we focused on the primal-dual formulation of entropy regularized Markov decision problems. We proposed a quadratically convexified primal-dual formulation that makes the landscape of the objective function smoother and enables faster numerical algorithms. We proved the equivalence of the quadratically convexified primal-dual formulation with the original primal-dual formulation. Leveraging the enhanced convexity of the objective function, we proposed an NGAD method and proved its convergence properties using the Lyapunov methods. We further introduced an INGAD algorithm that accelerates convergence significantly. The efficiency and robustness of the proposed algorithms are demonstrated through multiple numerical experiments. 

For future directions, one can potentially extend the convergence analysis here to the finite sample case with standard statistical methods. Another interesting direction to explore is the application of other optimization techniques to the convexified formulation proposed here. 

\section{Proofs}\label{sec:proofs}

\subsection{Proof of \Cref{thm:equi}}

\begin{proof}[Proof of \Cref{thm:equi}]
First, we show that there exists a unique solution to \eqref{eq:pdstandard}. By
\cite{geist2019theory}, there exists a unique optimal policy $\pi^*$ and a unique optimal value
function $v^* = v_{\pi^*}$ such that \eqref{eq:optimality}, or equivalently, \eqref{eq:primal} and
\eqref{eq:vtopi} hold. From \cite{ying2020note}, we know that this optimal value function and policy
$(v^*, \pi^*)$ also yields a solution $(v^*, u^\circ)$ to the primal-dual problem by $u_{sa}^\circ = \pi_{sa}^*(K_{\pi^*}^{-\T}e)_s$. Also from \cite{ying2020note} we know that any solution to the primal-dual formulation \eqref{eq:pdstandard} satisfies $v = v^*$, $u_{sa}/\bu_s=\pi_{sa}^*, (s,a)\in S\times A$ and $\bu = K_{\pi^*}^{-\T}e$, which combined with the uniqueness of $(v^*, \pi^*)$ shows that the solution $(v^*, u^\circ)$ to \eqref{eq:pdstandard} is unique.

Next, we show that $(v^*, u^*)$ satisfies the first-order condition of \eqref{eq:strict}, where 
\eq{\label{eq:defmu} u_{sa}^* := \frac{w_s}{\bu_s^\circ}u_{sa}^\circ, ~w := \alpha K_{\pi^*}^{-\T}v^*.}
The first-order condition of \eqref{eq:pdstandard} gives
\eq{
\begin{aligned}
\label{eq:standard_1st}
  e_{s'} - \sum_{sa} K_{ass'} u_{sa}^\circ &= 0, \quad\forall s'\in S,\\
  r_{sa} - \sum_{s'} K_{ass'} v_{s'}^* - \tau \log (u_{sa}^\circ/\bu_s^\circ)  &=0, \quad\forall (s, a)\in S\times A.
\end{aligned}
}
Since $e_{s'} = \sum_{sa} K_{ass'} u_{sa}^\circ=\sum_s\sum_a K_{ass'} \pi_{sa}^*\bu_s^\circ=\sum_s K_{\pi^* ss'}\bu_s^\circ$,
we have $\bu^\circ = K_{\pi^*}^{-\T}e = e+\sum_{k=1}^{\infty}\gamma^k (P_{\pi^*}^\T)^k e$, 
and thus $\bu_s^\circ\geq e_s>0$ for all $s\in S$. Similarly, $w_s\geq\alpha v_s^*$ for all $s\in S$ since $w = \alpha K_{\pi^*}^{-\T}v^*$. By \eqref{eq:vtopi}, it is also known that $\pi_{sa}^* > 0$
for all $(s, a)\in S\times A$, so $(r_{\pi^*}-\tau h_{\pi^*})_s>0$ for all $s\in S$ since $r$ is
nonnegative. Again by an expansion of $K_{\pi^*}^{-1}$, one can show that $v_s^* = K_{\pi^*}^{-1}(r_{\pi^*}-\tau h_{\pi^*}) \geq (r_{\pi^*}-\tau h_{\pi^*})_s>0$.
Hence $ u_{sa}^* = \frac{w_s}{\bu_s^\circ}u_{sa}^\circ>0$ is well-defined. In addition, $\bu^* = w$ and
\eq{\label{eq:samepolicy} \frac{u_{sa}^*}{\bu_{sa}^*}=\frac{u_{sa}^\circ}{\bu_s^\circ} = \pi_{sa}^*,}
As a result,
\eq{\label{eq:replace1} r_{sa} - \sum_{s'} K_{ass'} v_{s'}^* = \tau \log
  \frac{u_{sa}^\circ}{\bu_s^\circ} =\tau \log \frac{u_{sa}^*}{\bu_{sa}^*}, \quad\forall (s, a)\in
  S\times A.
}
Moreover, one can show that \eq{\label{eq:replace2}
\begin{aligned}
  \alpha v_{s'}^* &= K_{\pi^*}^\T w=\sum_s K_{\pi^* ss'}\frac{u_{sa}^*}{\pi_{sa}^*}=\sum_s\left(\sum_a K_{ass'} \pi_{sa}^*\right)\frac{u_{sa}^*}{\pi_{sa}^*} = \sum_{sa} K_{ass'} u_{sa}^*, \quad\forall s'\in S.
\end{aligned}
}
Combining \eqref{eq:replace1} and \eqref{eq:replace2}, we conclude that $(v^*, u^*)$ is a solution to 
\eq{\label{eq:strict1st}
\begin{aligned}
  \alpha v_{s'} - \sum_{sa} K_{ass'} u_{sa} &= 0, \quad\forall s'\in S,\\
  r_{sa} - \sum_{s'} K_{ass'} v_{s'} - \tau \log \frac{u_{sa}}{\bu_s} &=0, \quad\forall (s, a)\in S\times A.
\end{aligned}
} This is the first-order stationary condition for the problem \eqref{eq:strict}.

Finally, we show that $(v^*, u^*)$ is the unique solution to \eqref{eq:strict1st}.  Assume that $(v^1,u^1)$
and $(v^2, u^2)$ are both solutions to \eqref{eq:strict1st}. If $v^1\not=v^2$, then $E(v^1,
u^1)<E(v^2,u^1)$ and $E(v^2, u^2)<E(v^1, u^2)$ since for any $u$, $E(v, u)$ is strictly convex in $v$. On the other hand, for any $v$, $E(v, u)$ is concave in $u$ (see for
example \cite{neu2017unified} or \cite{ying2020note}). So $E(v^1, u^1)\geq E(v^1, u^2)$ and
$E(v^2,u^2)\geq E(v^2, u^1)$ and
\[
E(v^1, u^1)\geq E(v^1, u^2)>E(v^2, u^2)\geq E(v^2, u^1)>E(v^1, u^1),
\]
which is a contradiction, so we must have $v^1=v^2$ instead. By the second equation in \eqref{eq:strict1st},
\[
\frac{u^1_{sa}}{\bu^1_{s}} = \exp \left(\tau^{-1}(r_{sa} - \sum_{s'} K_{ass'} v^1_{s'})\right) =
\exp\left(\tau^{-1}\left(r_{sa} - \sum_{s'} K_{ass'} v^2_{s'}\right)\right) =
\frac{u^1_{sa}}{\bu^1_{s}},
\]
thus $\pi^1 = \pi^2$, where $\pi^1_{sa}=\frac{u^1_{sa}}{\bu^1_{s}}$, $\pi^2_{sa} = \frac{u^2_{sa}}{\bu^2_{s}}$. Since $(\pi^1, v^1)=(\pi^2, v^2)$, by the first equation in \eqref{eq:strict1st},
\[
\bu^1 = \alpha K_{\pi_1}^{-\T} v^1 = \alpha K_{\pi_2}^{-\T} v^2 = \bu^2,
\]
As a result, 
\[
u^1_{sa} = \bu^1_{s} \cdot \frac{u^1_{sa}}{\bu^1_{s}} = \bu^2_{s} \cdot \frac{u^2_{sa}}{\bu^2_{s}} =
u^2_{sa}, \quad \forall (s, a)\in S\times A,
\]
and $(v^1, u^1) = (v^2, u^2)$. Hence the solution to \eqref{eq:strict1st} is unique. Theerefore, $(v^*,u^*)$ is
the unique solution to \eqref{eq:strict1st}. By equation \eqref{eq:samepolicy}, the policy yielded
by $u^*$ coincides with the optimal policy $\pi^*$, which finishes the proof.

\end{proof}

\subsection{Proof of \Cref{lemma:L}}

\begin{proof}[Proof of \Cref{lemma:L}]
From the definition of $L$ we know that $\frac{\p^2 L}{\p u_{sa} \p v_{s'}} = 0$. Moreover, 
\[\begin{aligned}
   \frac{\p^2 L}{\p v_s \p v_{s'}} = \alpha \delta_{ss'},~\frac{\p^2 L}{\p u_{sa} \p u_{s'a'}} =\tau\delta_{(s,a),(s',a')}\frac{u_{sa}^*}{u_{sa}^2},  \quad (s, s', a, a')\in S^2\times A^2,
\end{aligned}
\]
which means that the Hessian matrix of $L$ is a diagonal matrix with positive diagonal elements on
the domain $\mathbb{R}^{|S|}\times\mathbb{R}_+^{|S|\times|A|}$. Hence $L$ is strictly convex. Since
the first-order condition: \eq{\label{eq:lyap1st}
\begin{aligned}
  \frac{\p L}{\p v_s} = \alpha (v_s - v_s^*) = 0, ~\frac{\p L}{\p u_{sa}} = \tau\left(1 - \frac{u_{sa}^*}{u_{sa}} \right) = 0, \quad (s, a) \in S\times A,
\end{aligned}
}
has a unique solution $(v, u) = (v^*, u^*)$, it is also the unique global minimum of $L$. Let
\[
\varphi_s(x) = \frac{1}{2}\alpha |x-v^*_s|^2, \quad \psi_{sa}(x)=\tau (u^*_{sa}\log u^*_{sa}/x + x-u^*_{sa}).
\]
By the calculation above, one can also show that $\varphi_s$ and $\psi_{sa}$ are strictly convex and
non-negative. Moreover, since $\underset{x\rightarrow+\infty}{\lim}\psi_{sa}(x) = +\infty$, we have $M(C) = \underset{sa}{\max}\sup\{x>0\mid \psi_{sa}(x) \leq C\}<+\infty$. 
As a result, the sublevel set
\[
\{(v, u)\mid L(v, u)\leq C\}\subset \{(v, u)\mid |v_{s'}-v_{s'}^*|\leq\sqrt{2C/\alpha},
0<u_{sa}<M(C), s'\in S, (s, a)\in S\times A\}
\]
is bounded.
\end{proof}

\subsection{Proof of \Cref{lemma:lyap}}
We first prove the following lemma.
\begin{lemma}
\label{lemma:helpconvex}
Define $H:\mathbb{R}_+^{|A|}\rightarrow \mathbb{R}$ by $H(z) = \sum_az_a\log
z_a-\bar{z}\log\bar{z}$, where $\bar{z} = \sum_az_a$, then $H$ is convex. Moreover,
$(z_1-z_2)\cdot(\grad H(z_1)-\grad H(z_2))\geq0$, and the equality is achieved if and only if $z_2 =
cz_1$ for some $c>0$.
\end{lemma}
\begin{proof}
The second-order derivatives of $H$ read $\frac{\p^2 H}{\p z_a \p z_{a'}} =
\frac{\delta_{aa'}}{z_a}-\frac{1}{\bar{z}}$. By the Cauchy-Schwarz inequality, for any $x\in \mathbb{R}^{|A|}$
\[
\begin{aligned}
  \sum_{aa'}x_a\frac{\p^2 H}{\p z_a \p z_{a'}}x_{a'}&=\sum_{aa'}x_a\left(\frac{\delta_{aa'}}{z_a}-\frac{1}{\bar{z}}\right)x_{a'}=\sum_a\frac{x_a^2}{z_a}-\frac{1}{\bar{z}}\left(\sum_ax_a\right)^2\geq 0.
\end{aligned}
\]
Hence the Hessian matrix of $H$ is positive semi-definite and $H$ is convex.  By convexity $(z_1-z_2)\cdot(\grad H(z_1)-\grad H(z_2))\geq0$.
Suppose now that equality holds. If $z_1 = z_2$, then clearly $z_2=cz_1$ for $c=1$. If
$z_1\not=z_2$, let $h(t) = H(z_1+t(z_2-z_1))$, then $h$ is also convex and $h'(0) = (z_2-z_1)\cdot
\grad H(z_1) = (z_2-z_1)\cdot \grad H(z_2) = h'(1)$, so $h'(t) = h'(0)$ for any $t\in[0, 1]$, 
thus
\[
0 = h''(0) = (z_2-z_1)^\T\grad^2 H(z_1)(z_2-z_1).
\]
Hence from the equality condition of the Cauchy-Schwarz inequality, we conclude
$z_2-z_1=\tilde{c}z_1$ and thus $z_2 = cz_1$ for some $c$, and we have $c>0$ since $z_1, z_2\in \mathbb{R}_+^{|A|}$.
\end{proof}

\begin{proof}[Proof of \Cref{lemma:lyap}]
By \Cref{thm:equi}, $(v^*, u^*)$ is also the unique solution to \eqref{eq:strict1st}, so
\eq{\label{eq:soln1st}
\begin{aligned}
  \alpha v^*_{s'} - \sum_{sa} K_{ass'} u^*_{sa} &= 0, \quad s'\in S,\\
  \left(r_{sa}-\sum_{s'} K_{ass'} v^*_{s'} \right) - \tau \log \frac{u^*_{sa}}{\bu^*_{s}} &=0, \quad (s, a)\in S\times A.
\end{aligned}
}
Subtracting this from the dynamics \eqref{eq:ANGflow1} leads to
\eq{\label{eq:relflow}
\begin{aligned}
  \dt{v_{s'}} & = - \left((v_{s'}-v_{s'}^*) - \frac{1}{\alpha} \sum_{sa} K_{ass'} (u_{sa}-u^*_{sa})\right), \quad s'\in S,  \\
  \dt{u_{sa}} & = - u_{sa} \left( \left(\log\frac{u_{sa}}{\bu_{s}}-\log\frac{u^*_{sa}}{\bu^*_{s}} \right) + \frac{1}{\tau}\sum_{s'} K_{ass'} (v_{s'}-v_{s'}^*) \right), \quad (s, a)\in S\times A. 
\end{aligned}
}
Taking the derivative of $L$ gives
\eq{\label{eq:dtL}
\begin{aligned}
  \dt{L} =& - \alpha \sum_{s'} (v_{s'} - v^*_{s'}) \left((v_{s'}-v^*_{s'}) - \frac{1}{\alpha} \sum_{sa} K_{ass'} (u_{as}-u^*_{as}) \right)\\
  &-\tau \sum_{sa} \frac{u_{sa}- u^*_{sa}}{u_{sa}}\cdot u_{sa} \left( \left(\log\frac{u_{sa}}{\bu_{s}}-\log\frac{u^*_{sa}}{\bu^*_{s}}\right) + \frac{1}{\tau}\sum_{s'} K_{ass'} (v_{s'}-v^*_{s'}) \right)\\
  =& -\alpha \sum_{s'} (v_{s'} - v^*_{s'})^2 - \tau \sum_{sa} (u_{sa}- u^*_{sa}) \left(\log\frac{u_{sa}}{\bu_{s}}-\log\frac{u^*_{sa}}{\bu^*_{s}}\right),
\end{aligned}
}
where we have used \eqref{eq:lyap1st}. By \Cref{lemma:helpconvex},
\[
\sum_{sa}(u_{sa}- u^*_{sa}) (\log u_{sa}/\bu_{s}-\log u^*_{sa}/\bu^*_{s}) = \sum_s
(u_s-u_s^*)\cdot(\grad H(u_s)-\grad H(u_s^*)) \geq 0,
\]
where $H$ is defined in \Cref{lemma:helpconvex}. Therefore,
\[
-\alpha \sum_{s'} (v_{s'} - v^*_{s'})^2 - \tau \sum_{sa} (u_{sa}- u^*_{sa}) \left(\log\frac{u_{sa}}{\bu_{s}}-\log\frac{u^*_{sa}}{\bu^*_{s}}\right)\leq0.
\]
By \Cref{lemma:helpconvex} the equality holds only when $v=v^*$ and $u_{sa} = c_su_{sa}^*$ for
$c_s>0$, $(s, a)\in S\times A$.  Let
\[
\mathcal{R} = \{(v, u)\mid -\alpha \sum_{s'} (v_{s'} - v^*_{s'})^2 - \tau \sum_{sa}
(u_{sa}-u^*_{sa}) (\log u_{sa}/\bu_{s}-\log u^*_{sa}/\bu^*_{s})=0\},
\]
then $\mathcal{R} = \{(v, u)\mid v = v^*, u_{sa} = c_su_{sa}^*, c_s\in\mathbb{R}_+, s\in S\}$. We
proceed to prove that the only trajectory of \eqref{eq:relflow} in $\mathcal{R}$ is
$(v,u)=(v^*,u^*)$. Since $v=v^*$ for any $(v,u)\in\mathcal{R}$, $\dt{v_{s'}}=0$ for any $s'\in S$.
The following equality
\eq{\label{eq:nonsing}
  \begin{aligned}
    0 & = \sum_{sa} K_{ass'} (u_{sa}-u^*_{sa}) = \sum_{sa}
    K_{ass'} (c_s-1)u^*_{sa} \\
    & = \sum_{sa} K_{ass'}  (c_s-1)\bu^*_s\pi_{sa}^*=\sum_sK_{\pi^*ss'}(c_s-1)\bu^*_s
  \end{aligned}
} means that, for any point $(v,u)$ on the trajectory of \eqref{eq:relflow} in $\mathcal{R}$,
$K_{\pi^*}^{\T} ((c-1)\bu^*) = 0$. Here $(c-1)\bu^*$ is the vector with length $|S|$ whose $s$-th
element is $(c_s-1)\bu^*_s$. Thus $c_s = 1$ for any $s\in S$, and the trajectory is a single point $(v, u) = (v^*, u^*)$.
\end{proof}

\subsection{Proof of \Cref{thm:linrate}}

\begin{proof}[Proof of \Cref{thm:linrate}]
The linearized dynamic of the standard natural gradient \eqref{eq:ANGflow1} is
\eq{\label{eq:linflow}
\begin{aligned}
  \dt{v_{s'}} & = - \left(((v_{s'}-v^*_{s'}) - \frac{1}{\alpha} \sum_{sa} K_{ass'} (u_{sa}-u^*_{sa})\right), \quad s'\in S,  \\
  \dt{u_{sa}} & = - u_{sa}^* \left( \frac{1}{\tau}\sum_{s'} K_{ass'} (v_{s'}-v^*_{s'}) + \frac{u_{sa}-u_{sa}^*}{u_{sa}^*} -\frac{\bu_s-\bu_s^*}{\bu_s^*}\right), \quad (s, a)\in S\times A. 
\end{aligned}
} Define matrix $\tilde{K}$ by $\tilde{K}_{(s-1)|A|+a, s'} = K_{ass'}$, and let $\delta v=v-v^*$, $\delta u=u-u^*$. Then \eqref{eq:linflow} becomes
\eq{\label{eq:linmat} \frac{\mathrm{d}}{\mathrm{d}t}\begin{bmatrix} {\delta v}\\ {\delta u}
\end{bmatrix}
= -\begin{bmatrix}
I_{|S|}&-\frac{1}{\alpha}\tilde{K}^\T\\
\frac{1}{\tau}\diag{(u^*)}\tilde{K}& \diag{(u^*)}M
\end{bmatrix}
\begin{bmatrix}
\delta v\\
\delta u
\end{bmatrix},
} where $\diag{(u^*)}$ is a diagonal matrix whose $((s-1)|A|+a)$-th diagonal element is
$u_{sa}^*$. Here $M$ is a block-diagonal matrix defined as: \eq{\label{eq:M} M=
\begin{bmatrix}
M_1&&\\
&\ddots&\\
&&M_{|S|}
\end{bmatrix},\quad
M_s = \diag{\left((u_s^*)^{-1}\right)}-\frac{1}{\bu_s^*}\mathbf{1}_{|A|}\mathbf{1}_{|A|}^\T, \quad s\in S,
}
where $\diag{\left((u_s^*)^{-1}\right)}$ is a diagonal $|A|\times|A|$ matrix with the $a$-th diagonal element equal to $1/u_{sa}^*$. Notice that $M$ is symmetric and by the Cauchy-Schwarz inequality, for any $x\in \mathbb{R}^{|A|}$
\eq{\label{eq:Mpos}
\begin{aligned}
  \sum_{aa'}x_aM_sx_{a'}&=\sum_a\frac{x_a^2}{u_{sa}^*}-\frac{1}{\bar{u}_s^*}\left(\sum_ax_a\right)^2\geq0, \quad s\in S.
\end{aligned}
}Hence $M_s$ is positive semi-definite for all $s$ and thus $M$ is also positive semi-definite.
Define invertible matrix $P$ as
\[
P = \begin{bmatrix}
\sqrt{\tau}I_{|S|}&\\
& \sqrt{\alpha}\diag{\left(\sqrt{u^*}\right)}
\end{bmatrix},
\]
where $\diag{\left(\sqrt{u^*}\right)}$ is a diagonal $|S||A|\times|S||A|$ matrix with the $((s-1)|A|+a)$-th diagonal element equal to $\sqrt{u_{sa}^*}$. Denote the matrix in the linearized dynamics \eqref{eq:linmat} as $-J$, i.e., 
\[J = \begin{bmatrix}
I_{|S|}&-\frac{1}{\alpha}\tilde{K}^\T\\
\frac{1}{\tau}\diag{(u^*)}\tilde{K}& \diag{(u^*)}M
\end{bmatrix}.
\]
Then
\[
\begin{aligned}
P^{-1}JP &= 
\begin{bmatrix}
\frac{1}{\sqrt{\tau}}I_{|S|}&\\
& \frac{1}{\sqrt{\alpha}}\diag{\left((\sqrt{u^*})^{-1}\right)}
\end{bmatrix}
\begin{bmatrix}
I_{|S|}&-\frac{1}{\alpha}\tilde{K}^\T\\
\frac{1}{\tau}\diag{(u^*)}\tilde{K}& \diag{(u^*)}M
\end{bmatrix}
\begin{bmatrix}
\sqrt{\tau}I_{|S|}&\\
& \sqrt{\alpha}\diag{\left(\sqrt{u^*}\right)}
\end{bmatrix}
\\
&=\begin{bmatrix}
I_{|S|}&-\frac{1}{\sqrt{\alpha\tau}}\tilde{K}^\T\diag{\left(\sqrt{u^*}\right)}\\
\frac{1}{\sqrt{\alpha\tau}}\diag{\left(\sqrt{u^*}\right)}\tilde{K}& \diag{\left(\sqrt{u^*}\right)}M\diag{\left(\sqrt{u^*}\right)}
\end{bmatrix}.
\end{aligned}
\]
It suffices to show that the real part of the
eigenvalues of $P^{-1}JP$ is positive. Denote $P^{-1}JP$ by $\tilde{J}$. Using the positive semi-definiteness of $M$, for any eigenpair $(\lambda, x)$ of $\tilde{J}$ we can deduce \eq{\label{eq:realpartineq}
\begin{aligned}
\operatorname{Re}(\lambda) =~& \frac{1}{2}\left(\frac{x^\Hlg\tilde{J}x}{x^\Hlg x}+\frac{x^\Hlg\tilde{J}^\Hlg x}{x^\Hlg x}\right)\\
=~& \frac{1}{2x^\Hlg x}x^\Hlg\Bigg(\begin{bmatrix}
I_{|S|}&-\frac{1}{\sqrt{\alpha\tau}}\tilde{K}^\T\diag{\left(\sqrt{u^*}\right)}\\
\frac{1}{\sqrt{\alpha\tau}}\diag{\left(\sqrt{u^*}\right)}\tilde{K}& \diag{\left(\sqrt{u^*}\right)}M\diag{\left(\sqrt{u^*}\right)}
\end{bmatrix}\\
&+
\begin{bmatrix}
I_{|S|}&\frac{1}{\sqrt{\alpha\tau}}\tilde{K}^\T\diag{\left(\sqrt{u^*}\right)}\\
-\frac{1}{\sqrt{\alpha\tau}}\diag{\left(\sqrt{u^*}\right)}\tilde{K}& \diag{\left(\sqrt{u^*}\right)}M\diag{\left(\sqrt{u^*}\right)}
\end{bmatrix}
\Bigg)x\\
=~&\frac{1}{x^\Hlg x}x^\Hlg\begin{bmatrix}
I_{|S|}&\\
& \diag{\left(\sqrt{u^*}\right)}M\diag{\left(\sqrt{u^*}\right)}
\end{bmatrix}x\\
\geq ~ & 0,
\end{aligned},
} where the superscript $\Hlg$ denotes the Hermitian transpose. Now we proceed to show $\operatorname{Re}(\lambda)\not=0$. Let $x = \begin{bmatrix} x_1 \\ x_2 \end{bmatrix}$, where
$x_1\in\mathbb{R}^{|S|}$, $x_2\in\mathbb{R}^{|S||A|}$. If $\operatorname{Re}(\lambda)=0$, then
\[
0 = x^\Hlg
\begin{bmatrix}
  I_{|S|}&\\
  & \diag{\left(\sqrt{u^*}\right)}M\diag{\left(\sqrt{u^*}\right)}
\end{bmatrix}x
= x_1^\Hlg x_1 + (\diag{(\sqrt{u^*})}x_2)^\Hlg M(\diag{(\sqrt{u^*})}x_2) \ge 0,
\]
thus $x_1 = 0$ and the equality condition of the Cauchy-Schwarz inequality \eqref{eq:Mpos} must
hold. Hence $(x_2)_{sa} = c_s\sqrt{u_{sa}^*}$ for some $c_s\in\mathbb{R}$, $s\in S$. We also know
that $c_s$ is not all zero for $s\in S$; otherwise, $x_2=0$ and $x=0$ is not an eigenvector. Thus
\[
\tilde{J}x = \begin{bmatrix}
I_{|S|}&-\frac{1}{\sqrt{\alpha\tau}}\tilde{K}^\T\diag{\left(\sqrt{u^*}\right)}\\
\frac{1}{\sqrt{\alpha\tau}}\diag{\left(\sqrt{u^*}\right)}\tilde{K}& \diag{\left(\sqrt{u^*}\right)}M\diag{\left(\sqrt{u^*}\right)}
\end{bmatrix}x 
=\frac{-1}{\sqrt{\alpha\tau}}\begin{bmatrix}
\tilde{K}^\T\diag{\left(\sqrt{u^*}\right)}x_2\\
0
\end{bmatrix},
\]
which is not a scalar multiple of $x$ unless $\tilde{K}^\T\diag{\left(\sqrt{u^*}\right)}x_2=0$. However, as
\[
\left(\tilde{K}^\T\diag{\left(\sqrt{u^*}\right)}x_2\right)_{s'} = \sum_{sa}K_{ass'}c_su_{sa}^* = \sum_s K_{\pi^*ss'}c_s\bu_s^*, \quad s'\in S,
\]
$\tilde{K}^\T\diag{\left(\sqrt{u^*}\right)}x_2 = K_{\pi^*}^\T c\bu^*$ where $c\bu^*$ denotes the elementwise product. Thus $K_{\pi^*}^\T c\bu^*=0$ and then $c\bu^*=0$, contradicting with the fact that $c_s$ is not all zero.
The contradiction means that $\operatorname{Re}(\lambda)\not=0$. Together with the inequality
\eqref{eq:realpartineq} we have $\operatorname{Re}(\lambda)>0$ for any eigenvalue $\lambda$ of
$J$. Hence $\operatorname{Re}(\lambda)<0$ for any eigenvalue $\lambda$ of $-J$, the matrix in the
linearized dynamics \eqref{eq:linmat}. By Lyapunov's indirect theorem \cite{haddad2011nonlinear},
\eqref{eq:ANGflow1} has locally exponential convergence.
\end{proof}

\subsection{Proof of \Cref{lemma:Lc}}

\begin{proof}[Proof of \Cref{lemma:Lc}]
Similar to \Cref{lemma:L}, we first note that $\frac{\p^2 L}{\p u_{sa} \p v_{s'}} = 0$. Moreover,
\[\begin{aligned}
   \frac{\p^2 L_c}{\p v_s \p v_{s'}} = \alpha \delta_{ss'}, ~\frac{\p^2 L_c}{\p u_{sa} \p u_{s'a'}} =\tau\delta_{ss'}\left(\delta_{aa'}\frac{u_{sa}^*}{u_{sa}^2}+\frac{c\bu_s^*}{(1-c)\bu_s^2}\right),  \quad (s, s', a, a')\in S^2\times A^2.
\end{aligned}
\]
Hence the Hessian matrix of $L_c$ is
\[
\begin{bmatrix}
\alpha I_{|S|}&0\\
0& \tau\diag{(u^*/u^2)} + \frac{c\tau}{1-c}B
\end{bmatrix},
\]
where $(u^*/u^2)_{sa} = u_{sa}^*/u_{sa}^2$ and $B$ is a positive definite block-diagonal matrix:
\eq{\label{eq:B}
B:=\begin{bmatrix}
\frac{\bu_1^*}{\bu_1^2}  \mathbf{1}_{|A|}\mathbf{1}_{|A|}^\T&&\\
&\ddots&\\
&&\frac{\bu_{|S|}^*}{\bu_{|S|}^2}  \mathbf{1}_{|A|}\mathbf{1}_{|A|}^\T
\end{bmatrix}.
} Thus the Hessian of $L_c$ is positive definite
and $L_c$ is strictly convex. The derivatives of $L_c$ are 
\eq{\label{eq:L1st}
\begin{aligned}
\frac{\p L_c}{\p v_s} &= \alpha (v_s - v_s^*), \quad s\in S,\\
  \frac{\p L_c}{\p u_{sa}} &= \tau \left( \frac{u_{sa}- u^*_{sa}}{u_{sa}} + \frac{c}{1-c}  \frac{\bu_s-\bu^*_s}{\bu_s} \right), \quad (s, a)\in S\times A,
\end{aligned}
}
from which we can see that $(v^*, u^*)$ is a solution to the first-order condition $\frac{\p L_c}{\p v} = 0$, $\frac{\p L_c}{\p u} = 0$. Since $L_c$ is strictly convex, it is also the unique minimizer of $L_c$. Now we prove that $L_c$ has bounded sublevel sets. Let $\ell(u) = \sum_s (\bu^*_{s}\log \bu^*_{s}/\bu_{s} + \bu_{s}-\bu^*_{s})$. Then $L_c(v, u) = L_0(v, u) + \frac{c\tau}{1-c}\ell(u)$. Since $\frac{\p^2 \ell}{\p u^2} = B$ is
positive definite, $\ell$ is strictly convex. Moreover, $\frac{\p \ell}{\p u_{sa}} =
(\bu_s-\bu^*_s)/\bu_s $ equals to $0$ when $u = u^*$, so by the strict convexity of $\ell$, $u^*$ is
the unique minimizer of $\ell$, and thus $\ell(u)\geq\ell(u^*)=0$. Hence the sublevel set $\{(v,u)\mid L_c(v, u)\leq C\}\subset \{(v, u)\mid L_0(v, u)\leq C\}$. Since the latter is bounded according to \Cref{lemma:L}, the sublevel set of $L_c$ is also
bounded.
\end{proof}

\subsection{Proof of \Cref{lemma:lyap2}}

\begin{proof}[Proof of \Cref{lemma:lyap2}]
Plugging the first-order condition \eqref{eq:soln1st} for the exact solution $(v^*, u^*)$ into the
interpolating natural gradient \eqref{eq:ANGflow2} results in \eq{
\label{eq:relflow2}
\begin{aligned}
  \dt{v_{s'}} & = - \left((v_{s'}-v^*_{s'}) - \frac{1}{\alpha} \sum_{sa} K_{ass'} (u_{sa}-u_{sa}^*) \right),\quad s'\in S,\\
  \dt{u_{s\cdot}} & = - \bu_s  \left(\diag(\pi_{s\cdot})-c\pi_{s\cdot}\pi_{s\cdot}^\T\right) \left(
  \left(\log \frac{u_{s\cdot}}{\bu_{s}}-\log\frac{u_{s\cdot}^*}{\bu_{s}^*}\right) + \frac{1}{\tau}\sum_{s'} K_{\cdot ss'} (v_{s'}-v^*_{s'}) \right),~ s\in S,
\end{aligned}
}
where $\pi_{sa}$ is defined as $u_{sa}/\bu_s$. 
A direct calculation shows that
\eq{
(u_{s\cdot}- u_{s\cdot}^*)/u_{s\cdot} + \frac{c}{1-c} (\bu_s-\bu^*_s)/\bu_s\mathbf{1}_{|A|} = \left( \diag(1/\pi_{s\cdot}) + \frac{c}{1-c} \mathbf{1}_{|A|} \mathbf{1}_{|A|}^\T \right)\left( \frac{u_{s\cdot}-u^*_{s\cdot}}{\bu_s}\right)
}
Then 
\eq{\label{eq:dtLc}
\begin{aligned}
  \dt{L} = - \alpha \sum_{s'}& (v_{s'} - v_{s'}^*) \left((v_{s'}-v_{s'}^*) - \frac{1}{\alpha} \sum_{sa} K_{ass'} (u_{sa}-u_{sa}^*) \right)\\
  -\tau \sum_{s}&\Bigg[ \left( \frac{u_{s\cdot }- u_{s\cdot }^*}{\bu_s}\right)^\T \left( \diag(1/\pi_{s\cdot}) + \frac{c}{1-c} \mathbf{1}_{|A|} \mathbf{1}_{|A|}^\T \right)
  \bu_s  \left(\diag(\pi_{s\cdot})-c\pi_{s\cdot}\pi_{s\cdot}^\T\right) \\
  &\left( \left(\log \frac{u_{s\cdot}}{\bu_{s}}-\log\frac{u^*_{s\cdot}}{\bu^*_{s}}\right) + \frac{1}{\tau}\sum_{s'} K_{\cdot ss'} (v_{s'}-v_{s'}^*) \right)\Bigg]\\
  =-\alpha \sum_{s'}& (v_{s'} - v_{s'}^*)^2 - \tau \sum_{sa} (u_{sa}- u^*_{sa}) \left(\log\frac{u_{as}}{\bu_{s}}-\log\frac{u^*_{as}}{\bu^*_{s}}\right),
\end{aligned}
}
where we have used the fact that
\[
\begin{aligned}
&\left( \diag(1/\pi_{s\cdot}) + \frac{c}{1-c} \mathbf{1}_{|A|} \mathbf{1}_{|A|}^\T \right)\left(\diag(\pi_{s\cdot})-c\pi_{s\cdot}\pi_{s\cdot}^\T\right)\\
=&~ \diag(1/\pi_{s\cdot})\diag(\pi_{s\cdot}) + \frac{c}{1-c} \mathbf{1}_{|A|} \mathbf{1}_{|A|}^\T\diag(\pi_{s\cdot})-\frac{c^2}{1-c} \mathbf{1}_{|A|} \mathbf{1}_{|A|}^\T\pi_{s\cdot}\pi_{s\cdot}^\T-c \diag(1/\pi_{s\cdot})\pi_{s\cdot}\pi_{s\cdot}^\T\\
=&~I+\left(\frac{c}{1-c}-\frac{c^2}{1-c}-c\right)\mathbf{1}_{|A|}\pi_{s\cdot}^\T = I \\
\end{aligned}
\]
Therefore, 
\[
\dt{L_c} = -\alpha \sum_{s'} (v_{s'} - v_{s'}^*)^2 - \tau \sum_{sa} (u_{sa}- u^*_{sa}) \left(\log\frac{u_{as}}{\bu_{s}}-\log\frac{u^*_{as}}{\bu^*_{s}} \right),
\]
where the right-hand side coincides with that of \eqref{eq:dtL}. Hence $\dt{L_c} = \dt{L_0} \leq 0$
by the proof of \Cref{lemma:lyap}.  Let
\[
\mathcal{R} = \{(v, u)\mid -\alpha \sum_{s'} (v_{s'} - v^*_{s'})^2 - \tau \sum_{sa} (u_{sa}- u^*_{sa}) (\log u_{sa}/\bu_{s}-\log u^*_{sa}/\bu^*_{s})=0\}.
\]
Then by the proof of \Cref{lemma:lyap}, $\mathcal{R} = \{(v, u)\mid v = v^*, u_{sa} =
c_su_{sa}^*, c_s\in\mathbb{R}_+, s\in S\}$. We proceed to prove that the only trajectory of
\eqref{eq:relflow2} in $\mathcal{R}$ is $(v, u)=(v^*, u^*)$. Since $v = v^*$ for any
$(v,u)\in\mathcal{R}$, $\dt{v_{s'}} = 0$ for $s'\in S$. In addition, for any $s'\in S$ we have
\[
0 = \sum_{sa} K_{ass'} (u_{sa}-u^*_{sa}) = \sum_sK_{\pi^*ss'}(c_s-1)\bu^*_s,
\]
by the same calculation as \eqref{eq:nonsing}.  This means that for point $(v, u)$ on the
trajectory of \eqref{eq:relflow2} in $\mathcal{R}$, $K_{\pi^*}^{\T} ((c-1)\bu^*) = 0$, thus $(c-1)\bu^* = 0$ and $c_s = 1$ for any $s\in S$. Since this is true for any $(v, u)$ on the trajectory, the trajectory is a single point
$(v, u) = (v^*, u^*)$.

\end{proof}

\subsection{Proof of \Cref{thm:linrate2}}

\begin{proof}[Proof of \Cref{thm:linrate2}]
The linearized dynamic of the interpolating natural gradient \eqref{eq:ANGflow2} is
\eq{\label{eq:linflow2}
\begin{aligned}
  \dt{v_{s'}} & = - \left(((v_{s'}-v^*_{s'}) - \frac{1}{\alpha} \sum_{sa} K_{ass'} (u_{sa}-u^*_{sa})\right), ~ s'\in S,  \\
  \dt{u_{s\cdot}} & = -\left(\diag(u_{s\cdot}^*)-\frac{c}{\bu_s^*}u_{s\cdot}^*(u_{s\cdot}^*)^\T\right) \left( \frac{1}{\tau}\sum_{s'} K_{\cdot ss'} (v_{s'}-v_{s'}^*) + \frac{u_{s\cdot}-u_{s\cdot}^*}{u_{s\cdot}^*} -\frac{\bu_s-\bu_s^*}{\bu_s^*}\mathbf{1}\right), ~ s\in S. 
\end{aligned}
} Define $\tilde{K}$ by $\tilde{K}_{(s-1)|A|+a, s'} = K_{ass'}$ and let $\delta v=v-v^*$, $\delta u=u-u^*$. Then \eqref{eq:linflow2} becomes
\eq{\label{eq:linmat2} \frac{\mathrm{d}}{\mathrm{d}t}\begin{bmatrix} {\delta v}\\ {\delta u}
\end{bmatrix}
= -\begin{bmatrix}
I_{|S|}&-\frac{1}{\alpha}\tilde{K}^\T\\
\frac{1}{\tau}G\tilde{K}& GM
\end{bmatrix}
\begin{bmatrix}
\delta v\\
\delta u
\end{bmatrix},
}
where $M$ is a block-diagonal matrix defined as in \eqref{eq:M} and $G$ is a block-diagonal matrix
\eq{\label{eq:G}
\begin{bmatrix}
G_1&&\\
&\ddots&\\
&&G_{|S|}
\end{bmatrix},
} with $G_s = \diag(u_{s\cdot}^*)-\frac{c}{\bu_s^*}u_{s\cdot}^*(u_{s\cdot}^*)^\T$. Notice that $G_s$
is symmetric. By the Cauchy-Schwarz inequality
\[
\begin{aligned}
x^\T G_sx &= \sum_a u_{sa}^*x_a^2-\frac{c}{\bu_s^*}\left(\sum_a u_{sa}^*x_a\right)^2\geq 
\frac{1}{\bu_s^*}\left(\sum_a u_{sa}^*x_a\right)^2-\frac{c}{\bu_s^*}\left(\sum_a u_{sa}^*x_a\right)^2\\
&=\frac{1-c}{\bu_s^*}\left(\sum_a u_{sa}^*x_a\right)^2>0, \quad\forall x\in\mathbb{R}^{|A|}, x\not=0,\quad \forall s\in S.
\end{aligned}
\]
Thus $G$ is positive definite, and we can define the positive definite square root $F$ of $G$, i.e., $F^2 = G$.
Define an invertible matrix $Q$
\[
Q = \begin{bmatrix}
\sqrt{\tau}I_{|S|}&\\
& \sqrt{\alpha}F
\end{bmatrix}
\]
and denote the matrix in the linearized dynamics \eqref{eq:linmat2} as $-J$, i.e.,
\[J = \begin{bmatrix}
I_{|S|}&-\frac{1}{\alpha}\tilde{K}^\T\\
\frac{1}{\tau}G\tilde{K}& GM
\end{bmatrix}.
\]
Then
\[
\begin{aligned}
Q^{-1}JQ = 
\begin{bmatrix}
\frac{1}{\sqrt{\tau}}I_{|S|}&\\
& \frac{1}{\sqrt{\alpha}}F^{-1}
\end{bmatrix}
\begin{bmatrix}
I_{|S|}&-\frac{1}{\alpha}\tilde{K}^\T\\
\frac{1}{\tau}G\tilde{K}& GM
\end{bmatrix}
\begin{bmatrix}
\sqrt{\tau}I_{|S|}&\\
& \sqrt{\alpha}F
\end{bmatrix}
=\begin{bmatrix}
I_{|S|}&-\frac{1}{\sqrt{\alpha\tau}}\tilde{K}^\T F\\
\frac{1}{\sqrt{\alpha\tau}}F\tilde{K}& FMF
\end{bmatrix}.
\end{aligned}
\]
It suffices to show that the real part of the eigenvalues of $Q^{-1}JQ$ is positive. Denote $Q^{-1}JQ$ by $\tilde{J}$. Using the positive semi-definiteness of $FMF$, for any eigenpair $(\lambda, x)$ of $\tilde{J}$ we can deduce \eq{\label{eq:realpartineq2}
\begin{aligned}
\operatorname{Re}(\lambda) =~& \frac{1}{2}\left(\frac{x^\Hlg\tilde{J}x}{x^\Hlg x}+\frac{x^\Hlg\tilde{J}^\Hlg x}{x^\Hlg x}\right)\\
=~& \frac{1}{2x^\Hlg x}x^\Hlg\Bigg(\begin{bmatrix}
I_{|S|}&-\frac{1}{\sqrt{\alpha\tau}}\tilde{K}^\T F\\
\frac{1}{\sqrt{\alpha\tau}}F\tilde{K}& FMF
\end{bmatrix}+
\begin{bmatrix}
I_{|S|}&\frac{1}{\sqrt{\alpha\tau}}\tilde{K}^\T F\\
-\frac{1}{\sqrt{\alpha\tau}}F\tilde{K}& FMF
\end{bmatrix}
\Bigg)x\\
=~&\frac{1}{x^\Hlg x}x^\Hlg\begin{bmatrix}
I_{|S|}&\\
& FMF
\end{bmatrix}x \geq  0.
\end{aligned}
}
It remians to show $\operatorname{Re}(\lambda)\not=0$. Let $x=\begin{bmatrix} x_1\\x_2 \end{bmatrix}$,
where $x_1\in\mathbb{R}^{|S|}$, $x_2\in\mathbb{R}^{|S||A|}$.
Then if $\operatorname{Re}(\lambda)=0$, 
\[
  0 = x^\Hlg\begin{bmatrix} I_{|S|}&\\ & FMF \end{bmatrix}x  =x_1^\Hlg x_1 + (Fx_2)^\Hlg M(Fx_2)\geq  0,
\]
Thus $x_1 = 0$, and the equality condition of the Cauchy-Schwarz inequality \eqref{eq:Mpos} must hold. Hence
$(Fx_2)_{sa} = c_su_{sa}^*$ for some $c_s\in\mathbb{R}$, $s\in S$. We also know that $c_s$ is not
all zero for $s\in S$; otherwise, $x_2=0$, so $x=0$ is not an eigenvector. Thus
\[
\tilde{J}x = \begin{bmatrix}
I_{|S|}&-\frac{1}{\sqrt{\alpha\tau}}\tilde{K}^\T F\\
\frac{1}{\sqrt{\alpha\tau}}F\tilde{K}& FMF
\end{bmatrix}x
=\frac{-1}{\sqrt{\alpha\tau}}\begin{bmatrix}
\tilde{K}^\T Fx_2\\
0
\end{bmatrix},
\]
which is not a scalar multiple of $x$ unless $\tilde{K}^\T Fx_2=0$. Since
\[
\left(\tilde{K}^\T Fx_2\right)_{s'} = \sum_{sa}K_{ass'}c_su_{sa}^* = \sum_s K_{\pi^*ss'}c_s\bu_s^*, \quad s'\in S,
\]
$\tilde{K}^\T Fx_2 = K_{\pi^*}^\T c\bu^*$. Thus $c\bu^*=0$, contradicting the fact that $c_s$ is not all zero. 
This contradiction means that $\operatorname{Re}(\lambda)\not=0$. Together with the inequality
\eqref{eq:realpartineq}, $\operatorname{Re}(\lambda)>0$ for any eigenvalue $\lambda$ of $J$. Hence
$\operatorname{Re}(\lambda)<0$ for any eigenvalue $\lambda$ of $-J$, the matrix in the linearized
dynamics \eqref{eq:linmat2}. Finally, by Lyapunov's indirect theorem \cite{haddad2011nonlinear},
\eqref{eq:ANGflow2} has locally exponential convergence. 
\end{proof}

\bibliographystyle{siamplain}
\bibliography{references}
\end{document}

%% file: ex_shared.tex
\usepackage{amsmath,amssymb,a4wide}
\usepackage{latexsym}
\usepackage{indentfirst}
\usepackage{graphicx}
\usepackage{placeins}
\usepackage{booktabs}
\usepackage{color}
\usepackage{algorithm}
\usepackage{algorithmic}
\usepackage{multirow}
\usepackage{enumerate}
\usepackage{subfigure}

\usepackage{lipsum}
\usepackage{amsfonts}
\usepackage{graphicx}
\usepackage{epstopdf}
\usepackage{algorithmic}
\ifpdf
  \DeclareGraphicsExtensions{.eps,.pdf,.png,.jpg}
\else
  \DeclareGraphicsExtensions{.eps}
\fi

\newsiamremark{remark}{Remark}
\newsiamremark{hypothesis}{Hypothesis}
\crefname{hypothesis}{Hypothesis}{Hypotheses}
\newsiamthm{claim}{Claim}

\headers{Accelerating Primal-dual for MDPs}{H. Li, H. Yu, L. Ying, and I. Dhillon}

\title{Accelerating Primal-dual Methods for Regularized Markov Decision Processes\thanks{Submitted to the editors DATE.
}}

\author{Haoya Li\thanks{Stanford University, Stanford, CA
  (\email{lihaoya@stanford.edu}).}
\and Hsiangfu Yu\thanks{Amazon Search 
  (\email{hsiangfu@amazon.com}).}
\and Lexing Ying\thanks{Stanford University, Stanford, CA
  (\email{lexing@stanford.edu}).}
\and Inderjit Dhillon\thanks{University of Texas at Austin and Google, work done while at Amazon Search 
  (\email{inderjit@cs.utexas.edu}).}}

\usepackage{amsopn}

\newcommand{\Hlg}{\mathrm{H}}
\newcommand{\diag}{\mathrm{diag}}

\newcommand{\eq}[1]{\begin{align}#1\end{align}}

\newcommand{\bbm}{\begin{bmatrix}}
\newcommand{\ebm}{\end{bmatrix}}
\newcommand{\dt}[1]{\frac{\mathrm{d}#1}{\mathrm{d}t}}
\newcommand{\R}{\mathrm{R}}

\newcommand{\T}{\mathsf{T}}
\newcommand{\p}{\partial}

\newcommand{\grad}{\nabla}

\newcommand{\bu}{\tilde{u}}
\newcommand{\g}{\gamma}


%% file: ex_article.bbl
\begin{thebibliography}{10}

\bibitem{agarwal2020optimality}
{\sc A.~Agarwal, S.~M. Kakade, J.~D. Lee, and G.~Mahajan}, {\em Optimality and
  approximation with policy gradient methods in {M}arkov decision processes},
  in {Conference on Learning Theory}, PMLR, 2020.

\bibitem{ahmed2019understanding}
{\sc Z.~Ahmed, N.~Le~Roux, M.~Norouzi, and D.~Schuurmans}, {\em Understanding
  the impact of entropy on policy optimization}, in {International Conference
  on Machine Learning}, PMLR, 2019.

\bibitem{asadi2017alternative}
{\sc K.~Asadi and M.~L. Littman}, {\em An alternative softmax operator for
  reinforcement learning}, in {International Conference on Machine Learning},
  PMLR, 2017.

\bibitem{bellman2015applied}
{\sc R.~E. Bellman and S.~E. Dreyfus}, {\em Applied dynamic programming},
  Princeton university press, 2015.

\bibitem{berkenkamp2017safe}
{\sc F.~Berkenkamp, M.~Turchetta, A.~P. Schoellig, and A.~Krause}, {\em Safe
  model-based reinforcement learning with stability guarantees},  (2017),
  \url{https://arxiv.org/abs/1705.08551}.

\bibitem{brockman2016openai}
{\sc G.~Brockman, V.~Cheung, L.~Pettersson, J.~Schneider, J.~Schulman, J.~Tang,
  and W.~Zaremba}, {\em Openai gym}, arXiv preprint arXiv:1606.01540,  (2016).

\bibitem{cen2020fast}
{\sc S.~Cen, C.~Cheng, Y.~Chen, Y.~Wei, and Y.~Chi}, {\em Fast global
  convergence of natural policy gradient methods with entropy regularization},
  July 2020, \url{https://arxiv.org/abs/2007.06558}.

\bibitem{cho2017deep}
{\sc W.~S. Cho and M.~Wang}, {\em Deep primal-dual reinforcement learning:
  Accelerating actor-critic using bellman duality}, Dec. 2017,
  \url{https://arxiv.org/abs/1712.02467}.

\bibitem{chow2018lyapunov}
{\sc Y.~Chow, O.~Nachum, E.~Duenez-Guzman, and M.~Ghavamzadeh}, {\em A
  lyapunov-based approach to safe reinforcement learning}, May 2018,
  \url{https://arxiv.org/abs/1805.07708}.

\bibitem{dai2018sbeed}
{\sc B.~Dai, A.~Shaw, L.~Li, L.~Xiao, N.~He, Z.~Liu, J.~Chen, and L.~Song},
  {\em Sbeed: Convergent reinforcement learning with nonlinear function
  approximation}, in {International Conference on Machine Learning}, PMLR,
  2018.

\bibitem{ding2020natural}
{\sc D.~Ding, K.~Zhang, T.~Basar, and M.~Jovanovic}, {\em Natural policy
  gradient primal-dual method for constrained markov decision processes}, in
  {Advances in Neural Information Processing Systems}, 2020.

\bibitem{geist2019theory}
{\sc M.~Geist, B.~Scherrer, and O.~Pietquin}, {\em A theory of regularized
  {M}arkov decision processes}, in {International Conference on Machine
  Learning}, PMLR, 2019.

\bibitem{gong2021primal}
{\sc H.~Gong}, {\em Primal-Dual Method for Reinforcement Learning and Markov
  Decision Processes}, PhD thesis, Princeton University, 2021.

\bibitem{haddad2011nonlinear}
{\sc W.~M. Haddad and V.~Chellaboina}, {\em Nonlinear dynamical systems and
  control}, Princeton university press, 2011.

\bibitem{kakade2001natural}
{\sc S.~M. Kakade}, {\em A natural policy gradient}, in {Advances in Neural
  Information Processing Systems}, 2001.

\bibitem{kalman1959control}
{\sc R.~Kalman and J.~Bertram}, {\em Control system analysis and design via the
  second method of {L}yapunov:(i) continuous-time systems (ii) discrete time
  systems}, IRE Transactions on Automatic Control, 4 (1959), pp.~112--112.

\bibitem{kalman1960control}
{\sc R.~E. Kalman and J.~E. Bertram}, {\em Control system analysis and design
  via the ``second method'' of {L}yapunov: I—continuous-time systems},
  Journal of Basic Engineering, 82 (1960), pp.~371--393.

\bibitem{khodadadian2021linear}
{\sc S.~Khodadadian, P.~R. Jhunjhunwala, S.~M. Varma, and S.~T. Maguluri}, {\em
  On the linear convergence of natural policy gradient algorithm}, in 2021 60th
  IEEE Conference on Decision and Control (CDC), IEEE, 2021, pp.~3794--3799.

\bibitem{lan2023policy}
{\sc G.~Lan}, {\em Policy mirror descent for reinforcement learning: Linear
  convergence, new sampling complexity, and generalized problem classes},
  Mathematical programming, 198 (2023), pp.~1059--1106.

\bibitem{lee2018stochastic}
{\sc D.~Lee and N.~He}, {\em Stochastic primal-dual {Q}-learning}, Oct. 2018,
  \url{https://arxiv.org/abs/1810.08298}.

\bibitem{li2021softmax}
{\sc G.~Li, Y.~Wei, Y.~Chi, Y.~Gu, and Y.~Chen}, {\em Softmax policy gradient
  methods can take exponential time to converge}, Feb. 2021,
  \url{https://arxiv.org/abs/2102.11270}.

\bibitem{li2021quasinewton}
{\sc H.~Li, S.~Gupta, H.~Yu, L.~Ying, and I.~Dhillon}, {\em Quasi-newton policy
  gradient algorithms}, Oct. 2021, \url{https://arxiv.org/abs/2110.02398}.

\bibitem{lyapunov1992general}
{\sc A.~M. Lyapunov}, {\em The general problem of the stability of motion},
  International journal of control, 55 (1992), pp.~531--534.

\bibitem{mei2020global}
{\sc J.~Mei, C.~Xiao, C.~Szepesvari, and D.~Schuurmans}, {\em On the global
  convergence rates of softmax policy gradient methods}, in {International
  Conference on Machine Learning}, PMLR, 2020.

\bibitem{meyn2012markov}
{\sc S.~P. Meyn and R.~L. Tweedie}, {\em Markov chains and stochastic
  stability}, Springer Science \& Business Media, 2012.

\bibitem{nachum2017trust}
{\sc O.~Nachum, M.~Norouzi, K.~Xu, and D.~Schuurmans}, {\em Trust-pcl: An
  off-policy trust region method for continuous control}, July 2017,
  \url{https://arxiv.org/abs/1707.01891}.

\bibitem{neu2017unified}
{\sc G.~Neu, A.~Jonsson, and V.~G{\'o}mez}, {\em A unified view of
  entropy-regularized {M}arkov decision processes}, May 2017,
  \url{https://arxiv.org/abs/1705.07798}.

\bibitem{perkins2002lyapunov}
{\sc T.~J. Perkins and A.~G. Barto}, {\em Lyapunov design for safe
  reinforcement learning}, Journal of Machine Learning Research, 3 (2002),
  pp.~803--832.

\bibitem{puterman2014markov}
{\sc M.~L. Puterman}, {\em {M}arkov decision processes: discrete stochastic
  dynamic programming}, John Wiley \& Sons, 2014.

\bibitem{rawlik2013stochastic}
{\sc K.~Rawlik, M.~Toussaint, and S.~Vijayakumar}, {\em On stochastic optimal
  control and reinforcement learning by approximate inference}, in
  {Twenty-third International Joint Conference on Artificial Intelligence},
  AAAI Press, 2013.

\bibitem{schulman2015trust}
{\sc J.~Schulman, S.~Levine, P.~Abbeel, M.~Jordan, and P.~Moritz}, {\em Trust
  region policy optimization}, in {International Conference on Machine
  Learning}, PMLR, 2015.

\bibitem{schulman2015high}
{\sc J.~Schulman, P.~Moritz, S.~Levine, M.~Jordan, and P.~Abbeel}, {\em
  High-dimensional continuous control using generalized advantage estimation},
  June 2015, \url{https://arxiv.org/abs/1506.02438}.

\bibitem{schulman2017proximal}
{\sc J.~Schulman, F.~Wolski, P.~Dhariwal, A.~Radford, and O.~Klimov}, {\em
  Proximal policy optimization algorithms}, July 2017,
  \url{https://arxiv.org/abs/1707.06347}.

\bibitem{serrano2020faster}
{\sc J.~B. Serrano and G.~Neu}, {\em Faster saddle-point optimization for
  solving large-scale {M}arkov decision processes}, in {Learning for Dynamics
  and Control}, PMLR, 2020.

\bibitem{sutton2018reinforcement}
{\sc R.~S. Sutton and A.~G. Barto}, {\em Reinforcement learning: An
  introduction}, MIT press, 2018.

\bibitem{sutton1999policy}
{\sc R.~S. Sutton, D.~A. McAllester, S.~P. Singh, and Y.~Mansour}, {\em Policy
  gradient methods for reinforcement learning with function approximation}, in
  {Advances in Neural Information Processing Systems}, 2000.

\bibitem{wang2017primal}
{\sc M.~Wang}, {\em Primal-dual {$\pi $} learning: Sample complexity and
  sublinear run time for ergodic {M}arkov decision problems}, Oct. 2017,
  \url{https://arxiv.org/abs/1710.06100}.

\bibitem{wang2020randomized}
{\sc M.~Wang}, {\em Randomized linear programming solves the {M}arkov decision
  problem in nearly linear (sometimes sublinear) time}, Mathematics of
  Operations Research, 45 (2020), pp.~517--546.

\bibitem{wang2016online}
{\sc M.~Wang and Y.~Chen}, {\em An online primal-dual method for discounted
  {M}arkov decision processes}, in {IEEE 55th Conference on Decision and
  Control}, IEEE, 2016.

\bibitem{williams1992simple}
{\sc R.~J. Williams}, {\em Simple statistical gradient-following algorithms for
  connectionist reinforcement learning}, Machine learning, 8 (1992),
  pp.~229--256.

\bibitem{ying2020note}
{\sc L.~Ying and Y.~Zhu}, {\em A note on optimization formulations of {M}arkov
  decision processes}, Communications in Mathematical Sciences,  (2021).

\bibitem{zhan2021policy}
{\sc W.~Zhan, S.~Cen, B.~Huang, Y.~Chen, J.~D. Lee, and Y.~Chi}, {\em Policy
  mirror descent for regularized reinforcement learning: A generalized
  framework with linear convergence}, May 2021,
  \url{https://arxiv.org/abs/2105.11066}.

\bibitem{zhang2021cautious}
{\sc J.~Zhang, A.~S. Bedi, M.~Wang, and A.~Koppel}, {\em Cautious reinforcement
  learning via distributional risk in the dual domain}, IEEE Journal on
  Selected Areas in Information Theory, 2 (2021), pp.~611--626.

\end{thebibliography}
